\documentclass[12pt,twoside,reqno]{amsart}
\pdfoutput=1

\usepackage[main=english]{babel}
\usepackage[autostyle,english=american]{csquotes}
\MakeOuterQuote{"}

\usepackage{amssymb,mathtools}
\usepackage{bm}
\usepackage{xcolor}
\usepackage{dirtytalk}

\usepackage[
  top=1.4in,
  bottom=1.4in,
  left=2.5cm,
  right=2.5cm,
  headheight=0pt,
  headsep=0pt
]{geometry}
\usepackage{amsthm,thmtools,amsmath}

\usepackage[style=numeric-comp,url=false,isbn=false,maxnames=6, backend=biber]{biblatex}
\usepackage[final]{graphicx}
\graphicspath{{/img}}
\usepackage{subcaption}
\usepackage{tikz}
\usetikzlibrary{math}
\usetikzlibrary{babel,cd,shapes,3d,arrows,decorations.markings}
\tikzcdset{arrow style=math font}
\usetikzlibrary{decorations.pathmorphing,shapes}
\usetikzlibrary{decorations.pathreplacing}
\usepackage[dvipsnames]{xcolor}

\tikzset{
    cross/.style=
    {cross out, draw, solid, thin, 
    minimum size=1.4*(#1-\pgflinewidth), 
    inner sep=0pt, outer sep=0pt, rotate=20},
    cross/.default={3}
} 

\tikzset{
    circlecross/.style= 
    {circle, minimum width=8pt, draw, inner sep=0pt, path picture={\draw (path picture bounding box.south east) -- (path picture bounding box.north west) (path picture bounding box.south west) -- (path picture bounding box.north east);}},
    circlecross/.default={3}
} 

\tikzset{
    mid arrow/.style=
    {postaction={decorate,decoration={markings,mark=at position .5 with {\arrow[xshift=2pt,#1]{stealth}}}}},
} 

\tikzset{
    crosshor/.style=
    {cross out, draw, solid, thin, 
    minimum size=1.4*(#1-\pgflinewidth), 
    inner sep=0pt, outer sep=0pt, rotate=0},
    crosshor/.default={3}
}

\usepackage[shortlabels]{enumitem}

\usepackage[final]{hyperref}
\usepackage[noabbrev,capitalize]{cleveref}
\creflabelformat{equation}{#2\textup{#1}#3}
\hypersetup{
  pdfcreator = {},
  pdfproducer = {}
}

\declaretheorem[numberwithin=section]{theorem}
\declaretheorem[sibling=theorem]{lemma, proposition, corollary}
\declaretheorem[sibling=theorem,style=definition]{definition}
\declaretheorem[sibling=theorem,style=remark]{remark, example}

\newtheorem{thmINTRO}{Theorem}

\newcommand\pmat[1]{\begin{pmatrix}#1\end{pmatrix}}

\def\C{\mathbb{C}}

\def\N{\mathbb{N}}

\def\R{\mathbb{R}}

\def\Z{\mathbb{Z}}

\def\ca{{\mathcal A}}

\def\cc{{\mathcal C}}
\def\cd{{\mathcal D}}
\def\ce{{\mathcal E}}
\def\cf{{\mathcal F}}

\def\co{{\mathcal O}}

\def\ct{{\mathcal T}}

\newcommand{\symp}{\operatorname{Symp}}
\newcommand{\Interior}{\operatorname{Int}}
\newcommand{\ham}{\operatorname{Ham}}
\newcommand{\Stair}{\operatorname{Stair}}

\newcommand{\Sp}{\operatorname{Sp}}
\newcommand{\vol}{\operatorname{Vol}}

\newcommand{\vis}{\mathrm{vis}}
\newcommand{\sect}{\mathrm{sec}}
\newcommand{\can}{\mathrm{can}}
\newcommand{\std}{\mathrm{st}}
\newcommand{\sing}{\mathrm{sing}}
\newcommand{\eq}{\mathrm{eq}}

\newcommand{\ATF}{\mathfrak{A}}

\begin{document}

\author[]{Nikolas Adaloglou}
\address{Nikolas Adaloglou, 
    imj-prg, 
    Sorbonne Université et Université Paris Cité, CNRS}
\email{adaloglou@imj-prg.fr} 

\author[]{Johannes Hauber}
\address{Johannes Hauber,
    imj-prg, 
    Sorbonne Université et Université Paris Cité, CNRS}
\email{johannes.hauber@unine.ch} 

\date{\today}

\title{Quantitative symplectic topology of Katok's examples and equivariant symplectic embeddings}

\begin{abstract}
    The Katok examples on $S^2$ are induced by Randers metrics obtained by perturbing the round metric by the standard rotational Killing field:
    allowing a scaling factor of the round metric gives a two-parameter family of Randers metrics $F_{\alpha,\beta}$, parametrised by positive real numbers $(\alpha,\beta)$.
    We compute the set of all $(\alpha,\beta)\in (0,2+\sqrt{3})^2$ for which $D^*(S^2,F^*_{\alpha,\beta})$, the unit codisc bundle with respect to $F^*_{\alpha,\beta}$, symplectically embeds into the round codisc bundle $D^*S^2$.
    This set has the structure of an infinite staircase.
    We also establish a dictionary between embeddings of these codisc bundles, singular $A_1$-ellipsoid embeddings, and $\Z_2$-equivariant ellipsoid embeddings, showing that these embedding problems are equivalent.
    Furthermore, we prove the analogous results for the $\R P^2$ case and discuss a broader family of examples for which this dictionary applies.
\end{abstract}

\maketitle
\pagestyle{plain}

\section{Introduction}
\label{sec:intro}

Symplectic embedding problems have been one of the central themes in symplectic geometry ever since Gromov's seminal paper \cite{Gro85}.
Already for seemingly very simple set-ups, such as computing the set of all symplectic ellipsoids that embed into a ball, these problems are very rich in structure.
For a thorough discussion and the relevant background see \cite{Sch18}.
The purpose of this paper is to show that the same phenomenon appears naturally in a different geometric setting: embedding problems for unit codisc bundles of Finsler metrics that induce Katok's examples \cite{Ka74}, and equivariant symplectic embeddings of ellipsoids into balls.

\vspace{4pt}
More concretely, from the Finsler perspective we study the two-parameter family of Randers metrics on $S^2$, obtained by perturbing a scaled round metric by the standard rotational Killing field.
These are the classical examples introduced by Katok.
From the symplectic viewpoint, a Finsler metric is encoded by its unit codisc bundle, and the co-geodesic flow is the Reeb flow on the boundary.
For two positive real numbers $(\alpha,\beta) \in \R_{>0}^2$ we define the Finsler metric $F_{\alpha,\beta}$ of Katok's examples.
Here, the parameters $\alpha$ and $\beta$ measure the two oriented equatorial widths of the dual unit ball, as explicated in \cref{rmk:ab_meaning} and shown in \cref{fig:codisc_Katok}.
In particular $F^*_{\alpha,\alpha}=\tau / \alpha \lvert\cdot\rvert_{g_0^*}$, where $\tau=2\pi$ and $g_0^*$ denotes the dual of the round metric on the sphere $S^2 \subseteq \R^3$ of radius~$1$.

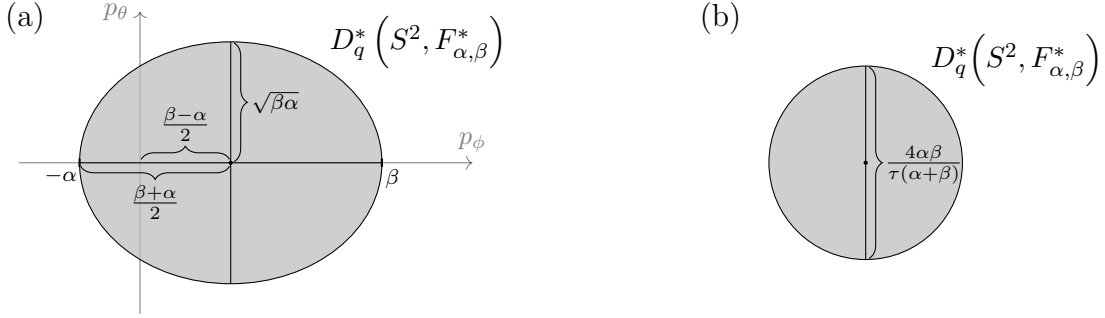
\begin{figure}[htb]
  \begin{center}   
    \begin{tikzpicture}[scale=0.8]
    
    \begin{scope}[shift={(-8,0)}]
        \node at (-1.9,2.4) {(a)};
        \node at (4.6,2) {\small $D^*_{q}\left(S^2,F_{\alpha,\beta}^*\right)$};
        \draw[opacity=0.5,->]
            (0,-2.5) -- (0,2.5);
        \draw[opacity=0.5,->]
            (-2,0) -- (5.5,0);
        \filldraw[fill=lightgray, fill opacity=0.75, draw=black] (3/2,0) ellipse (5/2 and 2);
        \node at (0,2.5) [left, opacity=0.5] {\footnotesize $p_\theta$};
        \node at (5.5,0) [above, opacity=0.5] {\footnotesize $p_\phi$};
        \node[xshift=-7, yshift=-5] at (-1,0) {\tiny $-\alpha$};
        \node[xshift=4, yshift=-6] at (4,0) {\tiny $\beta$};
        \draw [decorate, decoration={brace, amplitude=5pt, raise=0.2ex}] (0.02,0.01) -- (3/2-0.02,0.01) node[midway, yshift=2.8ex]{\footnotesize $\frac{\beta-\alpha}{2}$};
        \draw [decorate, decoration={brace, amplitude=5pt, raise=0.2ex}] (3/2-0.02,-0.01) -- (-1+0.02,-0.01) node[midway, yshift=-2.8ex]{\footnotesize $\frac{\beta+\alpha}{2}$};
        \draw [decorate, decoration={brace, amplitude=5pt, raise=0.2ex}] (3/2+0.01,2-0.02) -- (3/2+0.01,0.02) node[midway, xshift=3.2ex]{\tiny $\sqrt{\beta\alpha}$};
        \draw[thick] (-1,-0.07) -- (-1,0.07);
        \draw[thick] (4,-0.07) -- (4,0.07);
        \draw (-1,0) -- (4,0);
        \draw (3/2,2) -- (3/2,-2);
        \fill (3/2,0) circle[radius=1pt];
    \end{scope}

    \begin{scope}[shift={(4,0)}]
        \node at (-2.4,2.4) {(b)};
        \node at (2.5,1.7) {\small $D^*_{q}\Big(S^2,F_{\alpha,\beta}^*\Big)$};
        \filldraw[fill=lightgray, fill opacity=0.75, draw=black] (0,0) ellipse (8/5 and 8/5);
        \node at (0,2.5) [left, opacity=0.5] {};
        \node at (2.5,0) [above, opacity=0.5] {};
        \draw (0,-8/5) -- (0,8/5);
        \draw [decorate, decoration={brace, amplitude=5pt, raise=0.2ex}] (0.01,8/5-0.02) -- (0.01,-8/5+0.02) node[midway, xshift=4.2ex, yshift=-.1ex]{\footnotesize $\frac{4\alpha\beta}{\tau(\alpha+\beta)}$};
        \fill (0,0) circle[radius=1pt];
    \end{scope}
    
    \end{tikzpicture}
    \caption{(a) The codisc $D^*_{q}(S^2,F_{\alpha,\beta}^*)$ at a point $q\in S^2$ on the equator, where $F_{\alpha,\beta}$ is the Finsler metric inducing Katok's example. Here $p_\theta$ and $p_\phi$ are the cotangent coordinates to the volume normalised colatitude $\theta$ and longitude $\phi$. (b) The codisc $D^*_{q}(S^2,F_{\alpha,\beta}^*)$, where $q$ is either the north or the south pole. Note that the factor $\tau^{-1}$ appears because the figure is shown in an $g_0^*$-orthonormal frame. For the exact formulas see \cref{rmk:ab_meaning}}
    \label{fig:codisc_Katok}
  \end{center}
\end{figure}

\vspace{4pt}
Recall that $T^*S^2$ carries a canonical exact symplectic form $\omega_\can=d\lambda_\can$.
Therefore, the unit codisc bundles $D^*\left(S^2,F^*_{\alpha,\beta}\right) \subseteq T^*S^2$ are Liouville domains.
We ask whether 
$$
    D^*\left(S^2,F^*_{\alpha,\beta}\right) \xhookrightarrow{\raisebox{-0.5ex}{\( \hspace{0.5em}\scriptstyle s \hspace{0.5em} \)}} D^*\left(S^2,\tau \lvert\cdot\rvert_{g_0^*}\right),
$$ 
i.e.\ we wish to determine the set:
\begin{equation}\label{eq:E211_Intro}
    \ce_{2;1,1}:=\left\{(\alpha,\beta) \in \R_{>0}^2 \,\middle|\, D^*\left(S^2,F^*_{\alpha,\beta}\right) \xhookrightarrow{\raisebox{-0.5ex}{\( \hspace{0.5em}\scriptstyle s \hspace{0.5em} \)}} D^*\left(S^2,\tau \lvert\cdot\rvert_{g_0^*}\right) \right \}.
\end{equation}
Here and in the sequel, $\xhookrightarrow{\raisebox{-0.5ex}{\( \hspace{0.5em}\scriptstyle s \hspace{0.5em} \)}}$ means "there exists a symplectic embedding".\footnote{The convention that we will use is that for a closed domain $K \subseteq (M,\omega)$ of a symplectic manifold, a symplectic embedding $\iota:K \hookrightarrow (X,\sigma)$ means that it is the restriction of a symplectic extension $\hat{\iota}$ of $\iota$, i.e.\ that there exists a symplectic embedding $\hat{\iota}:U \hookrightarrow (X,\sigma)$, where $K \subseteq U$ is an open neighbourhood such that $\hat{\iota}|_K=\iota$.}
We will show in \cref{sec:Background} that $D^*\left(S^2,F^*_{\alpha,\beta}\right)$ is symplectomorphic to a natural almost toric domain $E_{2;1,1}(\alpha,\beta)$, a "$(2;1,1)$-ellipsoid", and then show  in \cref{sec:A211=E211} that $\ce_{2;1,1}$ coincides with the set 
$$
\ca_{2;1,1}:=\left\{(\alpha,\beta) \in \R_{>0}^2 \,\middle|\, E_{2;1,1}(\alpha,\beta) \xhookrightarrow{\raisebox{-0.5ex}{\( \hspace{0.5em}\scriptstyle s \hspace{0.5em} \)}} \C P^1 (1) \times \C P^1(1)\right \} \subseteq \R_{>0}^2,
$$
where $\C P^1 (1)$ is the $2$-sphere of area $1$.

\begin{thmINTRO}[\cref{sec:A211=E211}]
\label{thm:Intro_E211=A211}
    We have $\ce_{2;1,1}=\ca_{2;1,1}$.
\end{thmINTRO}

Recall that performing the symplectic cut that collapses the characteristics on the contact boundary $\partial D^*(S^2,\tau \lvert\cdot\rvert_{g_0^*})$ yields $\C P^1 (1) \times \C P^1(1)$.
This means that in order to show the equality $\ca_{2;1,1}=\ce_{2;1,1}$ it suffices to show that in the complement of a symplectic embedding $E_{2;1,1}(\alpha,\beta) \hookrightarrow \C P^1 (1) \times \C P^1(1)$ there always exists an embedded symplectic  $(+2)$-sphere.
This is shown by neck-stretching.
Classical arguments of Gromov \cite{Gro85} then show the uniqueness of such spheres in $\C P^1 (1) \times \C P^1(1)$ up to symplectic isotopy, which implies that we obtain an embedding 
$$
    D^*\left(S^2,F^*_{\alpha,\beta}\right) \xhookrightarrow{\raisebox{-0.5ex}{\( \hspace{0.5em}\scriptstyle s \hspace{0.5em} \)}} D^*\left(S^2,\tau \lvert\cdot\rvert_{g_0^*}\right).
$$

In order to formulate the symplectic staircase result we define the sequence
\begin{equation}
\label{eq:211_seq}
    \{m_i\}_{i \in \Z} \vcentcolon=\{\ldots,153,41,11,3,1,1,3,11,41,153,\ldots\}
\end{equation}
by the recursion $m_{i+2}=4m_{i+1}-m_i$ with initial conditions $m_0=1$ and $m_1=1$.
In particular, we have 
\begin{equation*}
    \lim_{i \to +\infty} \frac{m_{i+1}}{ m_i}= \lim_{i \to -\infty} \frac{m_{i}}{ m_{i+1}}=2+\sqrt{3}.
\end{equation*}
Define
\begin{equation*}
    \square_i(2;1,1):= \left(0,\frac{m_{i+1}}{m_i}\right) \times \left(0,\frac{m_{i}}{m_{i+1}}\right)
    \quad\text{and}\quad 
    \text{Stair}(2;1,1):=\bigcup_{i\in \Z} \square_i(2;1,1).
\end{equation*}

\begin{thmINTRO}[\cref{sec:A211_staircase}]
\label{thm:Intro_staircase}
    We have $\ca_{2;1,1} \cap (0,2+\sqrt{3})^2=\Stair(2;1,1)$.
\end{thmINTRO}

This theorem and its proof build heavily on the collaboration of the authors with Brendel, Evans and Schlenk \cite{ABEHS25}.
To show \cref{thm:Intro_staircase} one shows the uniqueness of symplectic embeddings $E_{2;1,1}(\alpha,\beta) \hookrightarrow \C P^1 (1) \times \C P^1(1)$ up to symplectomorphism and then uses visible symplectic curves to obstruct embeddings, in the spirit of \cite{McDSie25}.

\vspace{4pt}
The second novel aspect of this paper is that the two embedding problems $\ce_{2;1,1}$ and $\ca_{2;1,1}$, equivalent due to \cref{thm:Intro_E211=A211}, admit two further equivalent interpretations.
The almost toric domain $E_{2;1,1}(\alpha,\beta)$ is associated to the smoothing of the $A_1$-surface singularity $\mathbb C^2/\Z_2$, which carries a natural moment map.
Thus one obtains three parallel embedding problems for
\begin{equation*}
    \text{Katok's ellipsoids} 
    \quad\longleftrightarrow\quad
    \text{$A_1$-singular ellipsoids}
    \quad\longleftrightarrow\quad
    \text{$\Z_2$-equivariant ellipsoids}.
\end{equation*}

To make this dictionary precise we denote by $\ce^\sing_{2;1,1}$ the orbifold embedding problem for the $A_1$-surface singularity, and denote by $\ce^\eq_{2;1,1}$ the corresponding $\Z_2$-equivariant ellipsoid embedding problem in $\mathbb C^2$, that are defined as follows: recall that for $(\alpha,\beta) \in \R_{>0}^2$ the capacity normalised ellipsoid $E(\alpha,\beta)\subseteq \C^2$ is defined by
\begin{equation}\label{eq:capacity_normalised_ellipsoid}
    E(\alpha,\beta)\vcentcolon=\left\{(z_1,z_2) \in \C^2\, \middle|\, \frac{\pi}{\alpha} \lvert z_1\rvert^2 + \frac{\pi}{\beta} \lvert z_2 \rvert^2\leq 1\right\} \subseteq \C^2.
\end{equation}
The singular ellipsoid $E^\sing_{2;1,1}(\alpha,\beta)$ is then defined to be the quotient of $E(2\beta,2\alpha)$ by the diagonal $\Z_2$-action, i.e.\ $(z_1,z_2)\mapsto (-z_1,-z_2)$.
Since the standard symplectic form $\omega_\std$ on $\C^2$ descends to the singular quotient, there is a well-defined notion of symplectic orbifold embeddings $E^\sing_{2;1,1}(\alpha,\beta) \hookrightarrow B^\sing_{2;1,1}(1)\vcentcolon=E^\sing_{2;1,1}(1,1)$, which roughly means that the orbifold point goes to the orbifold point and the map is a symplectic embedding on the smooth part.
We will give a precise definition in \cref{subse:orbifold_objects}.

It follows essentially from the definitions that $\ce^\sing_{2;1,1}=\ce^\eq_{2;1,1}$, see \cref{lma:eq=sing}.
Both $\ce^\eq_{2;1,1}$ and $\ce^\sing_{2;1,1}$ are introduced in \cref{sec:Background}, see \cref{subse:orbifold_objects}.

\begin{thmINTRO}[\cref{sec:E211=E211sing}]
\label{thm:Intro_equality_211}
    We have $\ce_{2;1,1}=\ca_{2;1,1}=\ce^\sing_{2;1,1}=\ce^\eq_{2;1,1}$.
\end{thmINTRO}

Therefore, the infinite staircase of \cref{thm:Intro_staircase} simultaneously computes the Katok ellipsoid embedding problem, the singular embedding problem for the $A_1$-surface singularity, and the $\Z_2$-equivariant ellipsoid embedding problem.
To prove \cref{thm:Intro_equality_211}, it is enough to show the equality $\ce_{2;1,1}=\ce^\sing_{2;1,1}$. 
For this we use the nearby Lagrangian conjecture for $S^2$ to normalise symplectic embeddings $D^*(S^2,F^*_{\alpha,\beta}) \hookrightarrow D^*(S^2,\tau \lvert\cdot\rvert_{g_0^*})$ on the zero section and do the same by a Moser trick for singular embeddings. 
Having normalised the embeddings, the "cores" of the embeddings can be exchanged, proving the claim.
This is carried out in \cref{sec:E211=E211sing}.

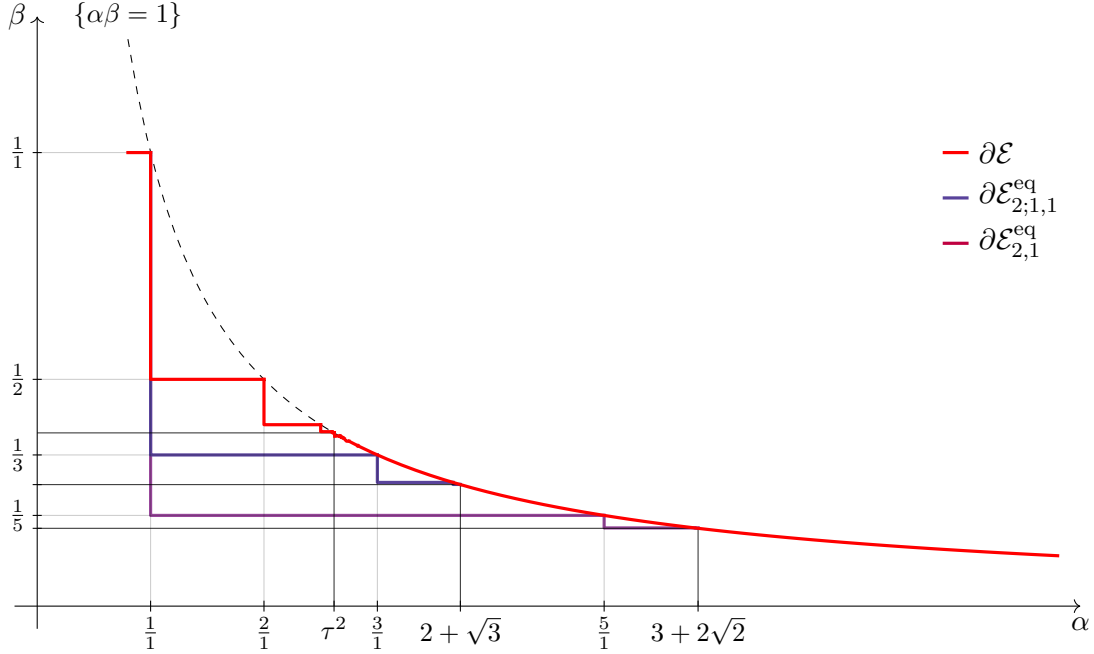
\begin{figure}[htb]
    \centering
    \begin{tikzpicture}[xscale=1,yscale=4,scale=1.5]
            \draw[very thick, red, line cap=rect, line join=round] (8,1) -- (8.2,1) 
                node[right] {\small \textcolor{black}{$\partial \ce$}};
            \draw[very thick, Violet, line cap=rect, line join=round] (8,0.9) -- (8.2,0.9) 
                node[right] {\small \textcolor{black}{$\partial \ce^{\eq}_{2;1,1}$}};
            \draw[very thick, purple, line cap=rect, line join=round] (8,0.8) -- (8.2,0.8) 
                node[right] {\small \textcolor{black}{$\partial \ce^{\eq}_{2,1}$}};
        
            \draw[->] (0,-0.05) -- (0,1.3) node[left] {\small $\beta$};
            \draw[->] (-0.2,0) -- (9.2,0) node[below] {\small $\alpha$};

            \draw[dashed,samples=100,domain=0.8:9,variable=\x] plot ({\x},{1/\x});
            \node[anchor=south] at (4/5,5/4) {\footnotesize $\{\alpha\beta=1\}$};
        
            \foreach \a/\b/\c in {
                1/1/5,
                1/5/29,
                5/29/169,
                29/169/985,
                169/985/5741
                }{
                \draw[very thick, Purple, line cap=rect, line join=round]
                  ({\b/\a},{\a/\b}) -- ({\b/\a},{\b/\c}) -- ({\c/\b},{\b/\c});
                }
                
                \draw (5,-0.012) -- (5,0.012);
                \node[below] at (5,0) {\footnotesize $\frac{5}{1}$};
                \draw (5.82843,-0.012) -- (5.82843,0.012);
                \node[below] at (5.82843,0) {\footnotesize $3+2\sqrt{2}$};
            
                \draw (-0.04,1/5) -- (0.04,1/5);
                \node[left] at (0,1/5) {\footnotesize $\frac{1}{5}$};
                \draw (-0.04,1/5.82843) -- (0.04,1/5.82843);
        
                \draw[draw opacity=0.2] (0,1/5) -- (5,1/5) -- (5,0);
                \draw[draw opacity=0.8] (0,1/5.82843) -- (5.82843,1/5.82843) -- (5.82843,0);
            
            \foreach \a/\b/\c in {
                1/1/3,
                1/3/11,
                3/11/41,
                11/41/153,
                41/153/571
                }{
                \draw[very thick, Violet, line cap=rect, line join=round]
                  ({\b/\a},{\a/\b}) -- ({\b/\a},{\b/\c}) -- ({\c/\b},{\b/\c});
                }
                
                \draw (3,-0.012) -- (3,0.012);
                \node[below] at (3,0) {\footnotesize $\frac{3}{1}$};
                \draw (3.73205,-0.012) -- (3.73205,0.012);
                \node[below] at (3.73205,0) {\footnotesize $2+\sqrt{3}$};
            
                \draw (-0.04,1/3) -- (0.04,1/3);
                \node[left] at (0,1/3) {\footnotesize $\frac{1}{3}$};
                \draw (-0.04,1/3.73205) -- (0.04,1/3.73205);
    
                \draw[draw opacity=0.2] (0,1/3) -- (3,1/3) -- (3,0);
                \draw[draw opacity=0.8] (0,1/3.73205) -- (3.73205,1/3.73205) -- (3.73205,0);          
                    
                \draw (1,-0.012) -- (1,0.012);
                \node[below] at (1,0) {\footnotesize $\frac{1}{1}$};
                \draw (2,-0.012) -- (2,0.012);
                \node[below] at (2,0) {\footnotesize $\frac{2}{1}$};
                \draw (2.618,-0.012) -- (2.618,0.012);
                \node[below] at (2.618,0) {\footnotesize $\tau^2$};
            
                \draw (-0.04,1) -- (0.04,1);
                \node[left] at (0,1) {\footnotesize $\frac{1}{1}$};
                \draw (-0.04,1/2) -- (0.04,1/2);
                \node[left] at (0,1/2) {\footnotesize $\frac{1}{2}$};

                \draw[draw opacity=0.2] (0,1) -- (1,1) -- (1,0);
                \draw[draw opacity=0.2] (0,1/2) -- (2,1/2) -- (2,0);
                \draw[draw opacity=0.8] (0,1/2.618) -- (2.618,1/2.618) -- (2.618,0);

            \foreach \a/\b/\c in {
                1/1/2,
                1/2/5,
                2/5/13,
                5/13/34,
                13/34/89,
                34/89/233,
                89/233/610,
                233/610/1597,
                610/1597/4181,
                1597/4181/10946
                }{
                \draw[very thick, red, line cap=rect, line join=round]
                  ({\b/\a},{\a/\b}) -- ({\b/\a},{\b/\c}) -- ({\c/\b},{\b/\c});
                }
            \draw[very thick, red, line cap=rect, line join=round] (0.8,1) -- (1,1);

            \draw[very thick, red, line cap=rect, line join=round] ({(3+sqrt(5))/2},{(3-sqrt(5))/2}) -- (21/8,3/8);
            \draw[very thick, red, line cap=rect, line join=round] (21/8,3/8) -- (8/3,3/8);

            \foreach \u/\v in {
                2.666666667/2.669268137,
                2.669270992/2.670826196,
                2.670837901/2.672603940,
                2.672620947/2.674217158,
                2.675000000/2.677032961,
                2.687500000/2.692307692,
                2.692907105/2.707106781,
                2.750000000/2.822875656,
                2.833333333/9
                }{
                \draw[very thick, red, line cap=rect, line join=round,samples=50,domain=\u:\v,variable=\x] plot ({\x},{1/\x});
                }
                
            \foreach \xL/\yL/\xC/\yC/\xR/\yR in {
                2.669268137/0.374634525/2.669268293/0.374634146/2.669270992/0.374634124,
                2.670826196/0.374415977/2.670826833/0.374414977/2.670837901/0.374414336,
                2.672603940/0.374166926/2.672605791/0.374164811/2.672620947/0.374164545,
                2.674217158/0.373941210/2.674335011/0.373831776/2.675000000/0.373831776,
                2.677032961/0.373547885/2.679069767/0.372093023/2.687500000/0.372093023,
                2.692307692/0.371428571/2.692307692/0.371352785/2.692907105/0.371345895,
                2.707106781/0.369398063/2.727272727/0.363636364/2.750000000/0.363636364,
                2.822875656/0.354248689/2.823529412/0.352941176/2.833333333/0.352941176
                }{
                \draw[very thick, red, line cap=rect, line join=round] (\xL,\yL) -- (\xC,\yC) -- (\xR,\yR);
                }
    \end{tikzpicture}
    \caption{Shown are the boundary curves of the sets of all $(\alpha,\beta)$ such that a) $E(\alpha,\beta)$ symplectically embeds into $B^4(1)$, which is the well-known Fibonacci staircase \cite{McDSch12} (here $\tau$ denotes the golden ratio); b) $E(\alpha,\beta)$ embeds $\Z_2$-equivariantly into $B^4(1)$ and c) $E(\alpha,\beta)$ embeds $\Z_4$-equivariantly into $B^4(1)$.}
    \label{fig:eq_embeddings}
\end{figure}

In analogy to the $S^2$-case, all of the above can also be shown for the $\R P^2$-case.
Recall that the symplectic cut construction on the contact boundary $\partial D^*(\R P^2,\tau \lvert\cdot\rvert_{g_0^*})$ yields $\C P^2 (1)$, where $\C P^2 (1)$ denotes the complex projective plane $\C P^2$ equipped with the Fubini--Study form $\omega_\text{FS}$ that gives a complex line area $1$.
Hence, we define in analogy to the $S^2$-case
$$
\ca_{2,1}:=\left\{(\alpha,\beta) \in \R_{>0}^2 \,\middle|\, E_{2,1}(\alpha,\beta) \xhookrightarrow{\raisebox{-0.5ex}{\( \hspace{0.5em}\scriptstyle s \hspace{0.5em} \)}} \C P^2 (2)\right \} \subseteq \R_{>0}^2,
$$
where $E_{2,1}(\alpha,\beta)$ is the almost toric domain that is symplectomorphic to the Katok ellipsoid $D^*\left(\R P^2,G^*_{\alpha,\beta}\right) \subseteq T^*\R P^2$, where $G^*_{\alpha,\beta}$ is induced by $F^*_{2\alpha,2\beta}$ on the quotient.
The target of the embedding problem is chosen to be $\C P^2 (2)$ in order to accommodate a global scaling factor, because $B_{2,1}(1) \cong D^*(\R P^2, \tau/2 \lvert \cdot \rvert^*_{g_0})$, where $B_{2,1}(1)\vcentcolon=E_{2,1}(1,1)$.

The staircase structure is governed by the numerology of the Pell-branch in the Markov tree, i.e.\ by the sequence 
\begin{equation*}
    \{n_i\}_{i \in \Z} \vcentcolon=\{\ldots,169,29,5,1,1,5,29,169,\ldots\},
\end{equation*}
where the recursion is given by $n_{i+2}=6n_{i+1} - n_i$, with initial conditions $n_0=1$ and $n_1=1$, which results in
\begin{equation*}
    \lim_{i \to +\infty} \frac{n_{i+1}}{n_i}= \lim_{i \to -\infty} \frac{n_{i}}{n_{i+1}}=3+2\sqrt{2}.
\end{equation*}
Defining
\begin{equation*}
    \square_i(2,1):= \left(0,\frac{n_{i+1}}{n_i}\right) \times \left(0,\frac{n_{i}}{n_{i+1}}\right)
    \quad\text{and}\quad 
    \text{Stair}(2,1):=\bigcup_{i\in \Z} \square_i(2,1)
\end{equation*}
we can state the results for the $\R P^2$-case.

\begin{thmINTRO}
\label{thm:intro_121}
    We have $\ca_{2,1} \cap (0,3+2\sqrt{2})^2=\Stair(2,1)$ and $\ce_{2,1}=\ca_{2,1}=\ce^\sing_{2,1}=\ce^\eq_{2,1}$.
\end{thmINTRO}

Here the sets $\ce_{2,1}$, $\ce^\sing_{2,1}$ and $\ce^\eq_{2,1}$ are also defined in analogy to what was discussed in the $S^2$-case.
The first part of this theorem, which establishes the symplectic staircase, was proven in \cite[Theorem 1.5.2]{ABEHS25}.\footnote{Note that \cite{ABEHS25} uses a slightly different normalisation. There the target is normalised to be $\C P^2(1)$, which means that the affine problem picks up a factor $2$.}
In \cref{sec:RP2_case}, we elaborate on the differences of the $S^2$-case to the $\R P^2$-case.
\cref{fig:eq_embeddings} shows the two staircases $\ca_{2,1}$ and $\ca_{2;1,1}$ together with the Fibonacci staircase $\ce$ by McDuff and Schlenk \cite{McDSch12}.

\subsection*{Open questions}
a) These results give rise to the following natural question.
Define the generalised equivariant embedding problems
$$
    \ce^\eq_{d;p,q}:=\left\{(\alpha,\beta) \in \R_{>0}^2 \,\middle|\, E(dp^2 \beta,dp^2 \alpha) \xhookrightarrow[\smash{\raisebox{0.6ex}{$\scriptstyle \eq$}}]{\hspace{0.5em} (d;p,q) \hspace{0.5em}} B^4(dp^2) \right \},
$$
for $d\in \N_{>0}$ and $0<q\leq p$ coprime integers.
For a detailed definition see \cref{sec:Background}.
Is it true that these sets exhibit infinite staircases? 

b) The staircases that are proven in this paper fit into the family of scalar $\Z_n$-actions on $\C^2$, meaning the action $\mu \cdot (z_1,z_2)=(\mu z_1, \mu z_2)$ for $\mu \in \bm{\mu}_n$ a $n$th root of unity. 
Is this so for all $n \in \N_{>0}$, i.e.\ do all $\Z_n$-equivariant embedding problems exhibit staircase behaviour and are they governed by the recursion 
$l_{i+2}=(n+2)l_{i+1}-l_i$?
The $n=1$ case is the Fibonacci staircase \cite{McDSch12} and the cases $n=2$ and $n=4$ are proven in \cref{thm:Intro_staircase} and \cref{thm:intro_121}.

c) The staircases proven in this paper do not address embeddings beyond the accumulation point.
Hence it would be interesting to see if there are "exceptional steps" in these staircases beyond the accumulation point.
See also \cite[Remark 1.5.3.(d)]{ABEHS25}.

\subsection*{Acknowledgments}
The authors would like to thank Ziad Chaoui, Jonny Evans and especially Felix Schlenk for helpful discussions.

JH acknowledges support by the Swiss National Science Foundation Postdoc.Mobility fellowship 242282.

\section{Background}
\label{sec:Background}

\subsection{Notation}

We abbreviate $\tau:=2 \pi$. 
$\C P^2(\lambda)$ for $\lambda \in \R_{>0}$ denotes $\C P^2$ equipped with the Fubini--Study form that gives a line area $\lambda$.
$\C P^1(\lambda)$ denotes $\C P^1$ equipped with the Fubini--Study form such that the area of $\C P^1$ is equal to $\lambda$. 
We will always assume that $(\alpha,\beta)\in \R^2_{>0}$ are two strictly positive real numbers.

\subsection{Orbifold objects}
\label{subse:orbifold_objects}

Given coprime integers $0<q\leq p$ and $d\in \N_{>0}$, we denote by $\bm{\mu}_{dp^2}$ the group of $dp^2$th roots of unity and by $\Gamma_{d;p,q}$ the Kähler action of $\bm{\mu}_{dp^2}$ on $\C^2$ with weights $(1,dpq-1)$, i.e.\ the roots $\mu \in \bm{\mu}_{dp^2}$ act on $(z_1,z_2) \in \C^2$ by 
\begin{equation}\label{eq:action_weights}
    \mu\cdot (z_1,z_2) = (\mu z_1, \mu^{dpq-1} z_2).
\end{equation}
The thereby defined cyclic quotient singularity $\C^2 /\Gamma_{d;p,q}$ is called a T-singularity and is usually denoted by $\frac{1}{dp^2}(1,dpq-1)$.
We adopt the convention to drop $d$ from the notation whenever it is equal to $1$ and that, whenever not mentioned otherwise, we are in the non-trivial case, i.e.\ $dp^2\neq 1$.

In \cite{Ev23:book,Ev24:KIAS,Sym03:four_two} it is explained how the toric geometry of the Milnor fibre $B_{d;p,q}$ of a T-singularity and of the T-singularity itself are related.
We will heavily exploit the following conceptual diagram
\begin{equation}\label{eq:diagram_relations}
    \text{objects in $B_{d;p,q}$} 
    \quad\longleftrightarrow\quad
    \text{objects in $ \C^2/\Gamma_{d;p,q}$}
    \quad\longleftrightarrow\quad
    \text{$\Gamma_{d;p,q}$-equivariant objects}.
\end{equation}

\begin{remark}
\label{rmk:moment_image_singular}
    Recall from \cite[Section 3.4]{Ev23:book} that there is moment map on $\frac{1}{dp^2}(1,dpq-1)$, whose moment image $\Delta_{d;p,q}$ is the wedge spanned by the vectors $(0,1)$ and $(dp^2,dpq-1)$.
    The usual moment map $\mu:\C^2 \to \R^2$ is given by
    $$
    \mu(z_1,z_2)=(\mu_1(z_1,z_2),\mu_2(z_1,z_2))=(\pi \lvert z_1\rvert^2, \pi \lvert z_2\rvert^2).
    $$
    This map descends to the quotient, since both Hamiltonians are invariant under $\Gamma_{d;p,q}$.
    However, the period lattice generated by this map on the quotient is no longer standard.
    In order to obtain a moment map on $\frac{1}{dp^2}(1,dpq-1)$ one needs to use the map induced by
    \begin{equation}\label{eq:Momentmap_cyclic}
        \left(\mu_2,\frac{1}{dp^2}\left(\mu_1+(dpq-1)\mu_2\right)\right)
    \end{equation}
    on the quotient.
    For $\alpha,\beta \in \R_{>0}$ define the singular ellipsoids $E^\sing_{d;p,q}(\alpha,\beta)$ via their toric base diagram 
    \begin{equation}\label{eq:def_Delta}
        \Delta_{d;p,q}(\alpha,\beta):=\text{Conv}\{(0,0),\alpha(dp^2,dpq-1),\beta(0,1)\}.
    \end{equation}
    This means that $E^\sing_{d;p,q}(\alpha,\beta)=E(dp^2 \beta, dp^2 \alpha) / \Gamma_{d;p,q}$, where $E(dp^2 \beta, dp^2 \alpha)$ denotes the capacity normalised ellipsoid in $\C^2$, as defined in \eqref{eq:capacity_normalised_ellipsoid}, because $E(dp^2 \beta, dp^2 \alpha)$ is mapped to $\Delta_{d;p,q}(\alpha,\beta)$ under the map defined in \eqref{eq:Momentmap_cyclic}.
\end{remark}

\begin{example}
    The two cases that we are going to consider in this paper are (a) $(d;p,q)=(2;1,1)$ and (b) $(d;p,q)=(1;2,1)$.
    In the case (a) the T-singularity is the $A_1$-surface singularity and its smoothing $B_{2;1,1}$ is well known to be symplectomorphic to $T^*S^2$.
    Therefore, the conceptual diagram \eqref{eq:diagram_relations} takes the form
    $$
        \text{objects in $T^*S^2$} 
        \quad\longleftrightarrow\quad
        \text{objects "in" the $A_1$-singularity}
        \quad\longleftrightarrow\quad
        \text{$\Z_2$-equivariant objects}.
    $$
    In the case (b) we have that $B_{2,1}$ is symplectomorphic to $T^*\R P^2$.
\end{example}

For an orbifold $X$ we denote its underlying topological space by $\lvert X \rvert$. 
If $Y$ is another orbifold, a \textit{reduced orbifold map} $f:X \to Y$ is given by a continuous map $\lvert f \rvert:\lvert X \rvert \to \lvert Y \rvert$ which admits a smooth local lift at each $x \in \lvert X \rvert$, meaning that there are:
\begin{enumerate}
    \item 
        orbifold charts $(\widetilde{U},\Lambda_x,\phi)$ and $(\widetilde{V},\Lambda_{\lvert f \rvert (x)},\psi)$ around $x$ and $\lvert f \rvert (x)$ such that $\rvert f \lvert(U)\subseteq V$,
    \item 
        a homomorphism $\overline{f}_x:\Lambda_x \to \Lambda_{\lvert f \rvert (x)}$
    \item 
        a smooth $\overline{f}_x$-equivariant map $\widetilde{f}_x: \widetilde{U} \to \widetilde{V}$ making the obvious diagram commute.
\end{enumerate}
Here $\Lambda_x$ denotes the local group at $x$, i.e.\ the isomorphism class of the isotropy subgroup; when necessary, we use the same notation
for a representative of this isomorphism class.

These definitions follow \cite{Car19}, which builds on the classical \cite{AdLeRu07}.
In the reduced category, the particular choices of local lifts and isotropy homomorphisms are not retained as part of the data.
An \textit{unreduced} orbifold map is a reduced one that remembers this data.
A reduced orbifold map $f:X \to Y$ is an \textit{orbifold embedding} if $\lvert f \rvert: \lvert X \rvert \to \lvert Y \rvert$ is a topological embedding and it admits smooth local lifts which are smooth embeddings.
The induced isotropy homomorphisms are then automatically injective.

A \textit{symplectic orbifold embedding} is an orbifold embedding which
admits symplectic local lifts.
 Thus it restricts to an ordinary
symplectic embedding on the regular loci and is locally represented near
every orbifold point by a symplectic equivariant embedding between
uniformising charts.
For a setting very similar to ours, see \cite{HiPiWu16}.
With this notion of symplectic orbifold embedding understood we define the set:
\begin{equation}\label{eq:set_singular_ellipsoid}
    \ce^\sing_{d;p,q}:=\left\{(\alpha,\beta) \in \R_{>0}^2 \,\middle|\, E^\sing_{d;p,q}(\alpha,\beta)  \xhookrightarrow{\raisebox{-0.5ex}{\( \hspace{0.5em}\scriptstyle s \hspace{0.5em} \)}} B^\sing_{d;p,q}(1) \right \},
\end{equation}
where $\xhookrightarrow{\raisebox{-0.5ex}{\( \hspace{0.5em}\scriptstyle s \hspace{0.5em} \)}}$ means "there exists a symplectic orbifold embedding" and, as usual, the singular ball is $B^\sing_{d;p,q}(1):=E^\sing_{d;p,q}(1,1)$.

By an equivariant symplectic embedding 
\begin{equation*}
    f:U \subseteq \C^2 \xhookrightarrow[\smash{\raisebox{0.6ex}{$\scriptstyle \eq$}}]{\hspace{0.5em} (d;p,q) \hspace{0.5em}} \C^2
\end{equation*}
we mean a symplectic $\Gamma_{d;p,q}$-equivariant embedding of a closed domain $U \subseteq \C^2$ that is invariant under the action $\Gamma_{d;p,q}$, i.e.\ $f$ satisfies $f(\mu z)=\mu f(z)$ for $\mu \in \bm{\mu}_{dp^2}$ and $z \in U$, where the action $\Gamma_{d;p,q}$ is the action of $\bm{\mu}_{dp^2}$ on $\C^2$ with weights $(1,dpq-1)$, introduced in \eqref{eq:action_weights}.
Set
\begin{equation}\label{eq:set_eq_ellipsoid}
    \ce^\eq_{d;p,q}:=\left\{(\alpha,\beta) \in \R_{>0}^2 \,\middle|\, E(dp^2 \beta,dp^2 \alpha) \xhookrightarrow[\smash{\raisebox{0.6ex}{$\scriptstyle \eq$}}]{\hspace{0.5em} (d;p,q) \hspace{0.5em}} B^4(dp^2) \right \},
\end{equation}
where the ellipsoid and the ball are just the usual, capacity normalised, domains in $\C^2$, as defined in \eqref{eq:capacity_normalised_ellipsoid}.
The following lemma is a direct consequence of the definitions.

\begin{lemma}
\label{lma:eq=sing}
    We have $\ce^\sing_{d;p,q}=\ce^\eq_{d;p,q}$.
\end{lemma}

\begin{proof}
    Abbreviate $N:=dp^2$ and $\Gamma=\Gamma_{d;p,q}$.
    Since every symplectic $\Gamma$-equivariant embedding $f:E(N\beta,N\alpha) \hookrightarrow B^4(N)$ descends to a symplectic orbifold embedding of the corresponding quotients by definition we have $\ce^\eq_{d;p,q} \subseteq \ce^\sing_{d;p,q}$.
    
    For the opposite inclusion assume that $f\colon E^\sing_{d;p,q}(\alpha,\beta) \hookrightarrow B^\sing_{d;p,q}(1)$ is a symplectic orbifold embedding.
    Since the orbifold point at the origin is the unique point with non-trivial isotropy, $f$ maps the singular origin to the origin. 
    Denote the ordinary covering maps induced by $\Gamma$ on the regular loci by 
    $$
        \pi_E:E(N\beta,N\alpha)\setminus \{0\} \to E_{d;p,q}^\sing(\alpha,\beta)\setminus \{0\}
        \quad\text{and}\quad 
        \pi_B:B(N)\setminus \{0\} \to B_{d;p,q}^\sing(1)\setminus \{0\}.
    $$
    Since the punctured ellipsoid $E(N\beta,N\alpha) \setminus \{0\}$ is simply connected the map $f \circ \pi_E$, admits a global lift:
    $$
    \widetilde{f}: E(N\beta,N\alpha) \setminus \{0\} \to B^4(N) \setminus \{0\}.
    $$
    By the definition of a symplectic orbifold embedding, there is also a symplectic local lift $\widetilde{f}_0$ of $f$ near the origin and an injective isotropy homomorphism $\rho:\Gamma \to \Gamma$ such that $\widetilde{f}_0(\mu z)=\rho(\mu)\widetilde{f}_0(z)$ for $\mu \in \bm{\mu}_N$, which implies that $\rho$ is an automorphism because the order of $\Gamma$ is finite.
    After composing $\widetilde{f}$ with a deck transformation, we may assume that it agrees with $\widetilde{f}_0$ on a punctured neighbourhood of the origin, and therefore it extends smoothly over the origin to a symplectic map, which we denote again by 
    $$
    \widetilde{f}: E(N\beta,N\alpha) \to B^4(N).
    $$
    Moreover, it is clear that $\widetilde{f}$ satisfies 
    \begin{equation}\label{eq:equivariance_origin}
        \widetilde{f}(\mu z)=\rho(\mu) \widetilde{f}(z)
    \end{equation}
    for $\mu \in \bm{\mu}_N$ and that $\widetilde{f}$ is a symplectic embedding because $\Gamma$ acts freely away from the origin and $\rho$ is injective.
    It remains to show that the twisting by the automorphism $\rho$ can be removed.
    Differentiating \eqref{eq:equivariance_origin} at the origin we obtain $L\mu=\rho(\mu)L$, where 
    $$L:=D\widetilde{f}(0) \in \mathrm{Sp}(4;\R).$$
    Now, we use the symplectic polar decomposition of $L$, see \cite[Remark 2.2.5]{McDSal16}.
    Since the action $\Gamma$ is unitary, both $\mu$ and $\rho(\mu)$ are unitary. 
    Define $U:=L (L^*L)^{-1/2} \in U(2)$.
    Then 
    $$
    \mu^*L^*L\mu=(L\mu)^*(L\mu)=(\rho(\mu)L)^* (\rho(\mu)L)=L^* \rho(\mu)^* \rho(\mu)L = L^*L
    $$
    so $L^*L$, and hence $(L^*L)^{-1/2}$, commutes with $\mu$, which implies
    \begin{equation}\label{eq:U_commutes}
        U\mu=L (L^*L)^{-1/2} \mu = L \mu (L^*L)^{-1/2}= \rho(\mu)U.
    \end{equation}
    Moreover, $U$ is the unitary factor in the symplectic polar decomposition of $L$, and therefore in particular preserves $B^4(N)$, and hence
    $$
    \widetilde{g}:= U^{-1} \circ \widetilde{f}: E(N\beta,N\alpha) \hookrightarrow B^4(N)
    $$
    is a well defined symplectic embedding that satisfies
    $$
    \widetilde{g}(\mu z)=U^{-1}(\widetilde{f}(\mu z))=U^{-1}\rho(\mu) \widetilde{f}(z)=\mu U^{-1} \widetilde{f}(z)= \mu\widetilde{g}(z).
    $$
    Therefore, $\ce^\sing_{d;p,q}\subseteq \ce^\eq_{d;p,q}$, finishing the proof.
\end{proof}

\begin{lemma}
\label{lma:normalise_sing_emb}
    Suppose that $(\alpha,\beta) \in \ce_{d;p,q}^\sing$.
    Then there exists a symplectic orbifold embedding $E^\sing_{d;p,q}(\alpha,\beta)  \xhookrightarrow{\raisebox{-0.5ex}{\( \hspace{0.5em}\scriptstyle s \hspace{0.5em} \)}} B^\sing_{d;p,q}(1)$ which agrees with the standard inclusion on $B^\sing_{d;p,q}(\varepsilon)$ for some small $\varepsilon > 0$.
\end{lemma}

\begin{proof}
    By \cref{lma:eq=sing} we have $(\alpha,\beta) \in \ce_{d;p,q}^\eq$, i.e.\ there exists a $\Gamma$-equivariant symplectic embedding $f:E(N\beta,N\alpha) \hookrightarrow B^4(N)$.
    In particular, $f(\mu z)=\mu f(z)$ and $f(0)=0$ since the origin is the unique fixed point of the action.
    We again consider the symplectic polar decomposition of $L\vcentcolon= Df(0) \in \Sp(4;\R)$. 
    Since $f$ is equivariant we have $\mu L=L \mu$.\\
    The first step of the argument is to construct an isotopy in 
    $$
    \Sp(4;\R)^\Gamma\vcentcolon= \left\{A\in\Sp(4;\R) \,\mid\, A \mu = \mu A \text{ for all } \mu \in \bm{\mu}_N \right\},
    $$
    connecting $L$ to the identity $I$.
    The polar composition yields
    $$
    L=UP,\quad
    U\vcentcolon=L (L^*L)^{-1/2}\in U(2), \quad
    P=P^* >0\quad
    \text{and} \quad P\in \Sp(4;\R).
    $$
    From the proof of \cref{lma:eq=sing}, specifically \eqref{eq:U_commutes}, we also have that both $U$ and $P$ commute with $\Gamma$, i.e.\ $U \in U(2)^\Gamma$ and $P \in \Sp(4;\R)^\Gamma$.
    To connect $P$ to the identity we define $P_t\vcentcolon=P^{1-t}$ for $t \in [0,1]$, which implies $P_t\in \Sp(4;\R)^\Gamma$.
    In order to connect $U$ to the identity in $U(2)^\Gamma$ one can just use the spectral theorem so diagonalise $U$ by a unitary matrix in order to connect $U$ to $I$ in $U(2)^\Gamma$, via $U_t$.
    Now, define $L_t=U_tP_t \in \Sp(4;\R)^\Gamma$, which is the desired equivariant isotopy from $L$ to $I$.
    The second step is now to use this isotopy to normalise $f$ at the origin in an equivariant manner, using Alexander's trick.
    In this way we find a $\Gamma$-equivariant symplectic embedding
    $E(N\beta, N\alpha) \hookrightarrow B^4(N)$ that agrees with the
    inclusion on $B^4(\varepsilon)$ for some $\varepsilon >0$.
    This embedding descends to the required orbifold embedding.
\end{proof}

\subsection{Ellipsoidal domains in cotangent bundles}

From now on we specialise to two cases: (a) $(d;p,q)=(2;1,1)$ and (b) $(d;p,q)=(1;2,1)$.
Moreover, we will concentrate on the case (a) whenever the discussion for (b) is analogous and elaborate on the differences only briefly.
As in \cite{ABEHS25, AdHa25} we define ellipsoid domains in $T^*S^2$ and $T^*\R P^2$ via their almost toric base diagrams. 
We assume working familiarity with almost toric geometry as explained in \cite{Sym03:four_two,Ev23:book,Ev24:KIAS}.
Recall that the Milnor fibre of the $A_1$-surface singularity $B_{2;1,1}$ admits an almost toric base diagram $\ATF_{2;1,1}$ and similarly the Milnor fibre of the $\frac{1}{4}(1,1)$-singularity, as explained in \cite[Chapter 7]{Ev23:book} and \cite{Ev24:KIAS}.
See \cref{fig:ATF_moment_map} for the almost toric base diagrams.

\begin{definition}
    For $\alpha,\beta \in \R_{>0}$ define the $(2;1,1)$-ellipsoid $E_{2;1,1}(\alpha,\beta) \subseteq B_{2;1,1}$ as the domain associated to the almost toric base diagram $\ATF_{2;1,1}(\alpha,\beta)$, as shown in \cref{fig:ATF_Bdpq}. 
    Analogously define $E_{2,1}(\alpha,\beta) \subseteq B_{2,1}$.
\end{definition}

\begin{figure}[htb]
  \begin{center}   
    \begin{tikzpicture}[scale=0.9]
    
    \begin{scope}[shift={(-8,0)}]
        \filldraw[lightgray,opacity=0.75] (0,1.5) -- (0,0) -- (3,1.5) -- cycle;
        \draw[thick] (0,1.5) -- (0,0) -- (3,1.5);
        \draw[dashed] (0.8,0.8) node[cross] {} -- (0.4,0.4) node[cross] {} -- (0,0);
        \node at (0.3,-1) {(a)};
    \end{scope}
    
    \begin{scope}[shift={(-5.3,-1)}]
        \filldraw[lightgray,opacity=0.75] (0,1) -- (0,0) -- (3,1.5) -- cycle;
        \draw[thick] (0,1) -- (0,0) -- (3,1.5);
        \draw [decorate,decoration={brace,amplitude=5pt,raise=1ex}] (0,0) -- (0,1) node[midway,xshift=-3ex]{\footnotesize $\beta$};
        \draw [decorate,decoration={brace,amplitude=5pt,raise=1ex}] (3,1.5) -- (0,0) node[midway,xshift=1.5ex,yshift=-2.5ex]{\footnotesize $\alpha$};
        \node at (3,1.5) [right] {\footnotesize $(2 \alpha,\alpha)$};
        \draw[dashed] (0.8,0.8) node[cross] {} -- (0.4,0.4) node[cross] {} -- (0,0);
        \draw[thick,dash pattern=on 7pt off 3pt] (0,1) -- (3,1.5);
    \end{scope}
    
    \begin{scope}[shift={(0.3,0)}]
        \node at (0.3,-1) {(b)};
        \filldraw[lightgray,opacity=0.75] (0,1.5) -- (0,0) -- (6,1.5) -- cycle;
        \draw[thick] (0,1.5) -- (0,0) -- (6,1.5);
        \draw[dashed] (1,0.5) node[cross] {} -- (0,0);
    \end{scope}
    
    \begin{scope}[shift={(2.6,-1)}]
        \filldraw[lightgray,opacity=0.75] (0,1) -- (0,0) -- (6,1.5) -- cycle;
        \draw[thick] (0,1) -- (0,0) -- (6,1.5);
        \draw [decorate,decoration={brace,amplitude=5pt,raise=1ex}] (0,0) -- (0,1) node[midway,xshift=-3ex]{\footnotesize $\beta$};
        \draw [decorate,decoration={brace,amplitude=5pt,raise=1ex}] (6,1.5) -- (0,0) node[midway,xshift=1.5ex,yshift=-2.5ex]{\footnotesize $\alpha$};
        \node at (6,1.5) [right] {\footnotesize $(4\alpha,\alpha)$};
        \draw[dashed] (1,0.5) node[cross] {} -- (0,0);
        \draw[thick,dash pattern=on 7pt off 3pt] (0,1) -- (6,1.5);
    \end{scope}
    
    \end{tikzpicture}
    \caption{(a) The almost toric base diagram $\ATF_{2;1,1}$ of $B_{2;1,1}$ and $\ATF_{2;1,1}(\alpha,\beta)$ of $E_{2;1,1}(\alpha,\beta)$. The branch cut points into the $(1,1)$-direction. The almost toric base diagram $\ATF_{2;1,1}$ is unbounded, whereas $\ATF_{2;1,1}(\alpha,\beta) \subseteq \ATF_{2;1,1}$ is the closed compact subset. The contact boundary $\partial E_{2;1,1}(\alpha,\beta)$, the lens space $L(2,1)$, i.e.\ $\R P^3$, of $E_{2;1,1}(\alpha,\beta)$ lives over the thickly dashed line.  (b) The analogous diagrams in the $(2,1)$-case.}
    \label{fig:ATF_Bdpq}
  \end{center}
\end{figure}
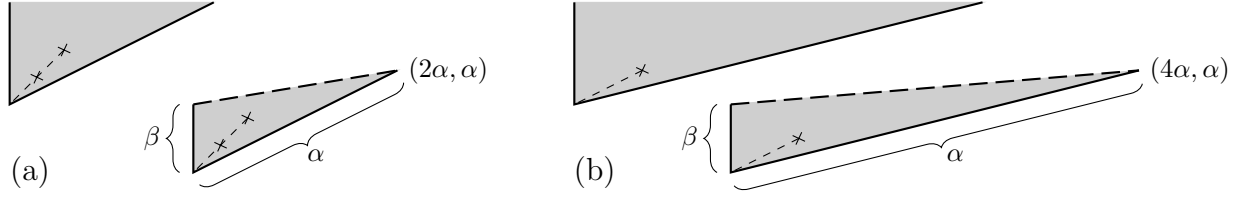

\begin{remark}
    These ellipsoids were defined for general $(p,q)$ by the authors in \cite{AdHa25}.
    A non-squeezing theorem for them was established in \cite{AdaBaHau26} and staircase theorems for them were proven in \cite{ABEHS25}.
    In \cref{cor:211_ellipsoids} we give a very concrete description of the $(2;1,1)$-ellipsoids in terms of the Finsler metrics that define Katok's examples. 
\end{remark}

\begin{remark}
    The placement of the nodes in the almost toric base diagrams $\ATF_{2;1,1}$ and $\ATF_{2,1}$ is not important for our purposes, since the symplectic manifolds associated to almost toric base diagrams related by nodal slides are symplectomorphic, as explained in \cite{Ev23:book,Sym03:four_two}.
\end{remark}

\begin{remark}
\label{rmk:moment_map_S2}
    In volume normalised spherical coordinates, i.e.\ colatitude $\theta \in (0,1/2)$, longitude $\phi \in \R / \Z$ and $\varphi(\theta,\phi)=(\sin(\tau \theta) \cos (\tau \phi), \sin(\tau \theta) \sin(\tau \phi), \cos(\tau \theta))$, the round metric and its dual are given by
    $$
    g_0 = \tau^2(d\theta^2 + \sin^2(\tau\theta)d\phi^2) 
    \qquad\text{and}\qquad
    g_0^* = \frac{1}{\tau^2}\left( \partial_\theta^2 + \frac{1}{\sin^2(\tau \theta)} \partial_\phi^2\right).
    $$
    Therefore, the length-normalised geodesic Hamiltonian $H$, as opposed to the kinetic-energy Hamiltonian, and the angular momentum $J$ are given by 
    \begin{equation}\label{eq:H_and_J}
        H= \tau \lvert \cdot \rvert_{g_0^*}
        \qquad\text{and}\qquad
        J= p_\phi
    \end{equation}
    and the map $\mu_{S^2}=(H,J)$ defines a moment map on $T^*S^2\setminus 0_{S^2}$.
    Its moment image is given by $\mu_{S^2}(T^*S^2\setminus 0_{S^2})=\{(x,y)\mid x > 0 \text{ and } \lvert y \rvert \leq x\}$.
\end{remark}

\begin{definition}
\label{def:Fab_metric}
    For $\alpha,\beta \in \R_{>0}$ define the Randers metric $F_{\alpha,\beta}$ on $TS^2$ via its dual
    \begin{equation}\label{eq:Randers_metric}
        F_{\alpha,\beta}^*:=\frac{\alpha+\beta}{2\alpha\beta}\left(H-\frac{\beta-\alpha}{\alpha+\beta} J\right)=\frac{\alpha+\beta}{2\alpha\beta}H - \frac{\beta-\alpha}{2\alpha\beta} J,
    \end{equation}
    where $H$ is the geodesic Hamiltonian and $J$ is the angular momentum introduced in \eqref{eq:H_and_J}.
\end{definition}

\begin{remark}\label{rmk:ab_meaning}
    Notice that the parameters $\alpha$ and $\beta$ measure how far the dual unit disc extends in the two oriented equatorial directions.
    Indeed, on the boundary ray $J=H$, away from the zero section, we have
    $$
    0=p_\theta^2 + p_\phi^2\left(\frac{1}{\sin^2(\tau \theta)}-1\right),
    $$
    which implies $p_\theta=0$ and $\sin(\tau \theta)=1$, i.e.\ $\theta=1/4$, means that on the equator we have
    $F_{\alpha,\beta}^*=H/\beta$,
    while on the boundary ray $J=-H$ we have $F_{\alpha,\beta}^*=H/\alpha$.
    In particular the unit codisc intersects the angular momentum axis in the interval $-\alpha\leq p_\phi \leq \beta$.
    
    Now, fix a point $q=(1/4,\phi_0)$ on the equator.
    This means that
    $$
    H(q,p)=\sqrt{p_\theta^2 + p_\phi^2} \quad\text{and}\quad J(p,q)=p_\phi.
    $$
    The unit codisc at $q$ is therefore given by
    $$
    D_q^*(S^2, F_{\alpha,\beta}^*)=\left\{ (p_\theta,p_\phi) \,\middle|\, (\alpha + \beta) \sqrt{p_\theta^2 + p_\phi^2} - (\beta-\alpha) p_\phi \leq 2 \alpha\beta \right\}
    $$
    and rewriting the defining inequality we obtain
    $$
    \frac{p_\theta^2}{\alpha\beta} + \frac{\big(2 p_\phi-(\beta-\alpha)\big)^2}{(\alpha+\beta)^2} \leq 1,
    $$
    which means that the unit codisc is an ellipse with centre $(0,(\beta-\alpha)/2)$, $p_\theta$-semiaxis equal to $\sqrt{\alpha \beta}$ and $p_\phi$-semiaxis equal to $(\alpha+\beta)/2$.
    
    Also notice that $F_{\alpha,\beta}^*$ is a drift perturbation of a scaling of the round cometric by the rotational Killing field.
\end{remark}

\begin{remark}
    The cogeodesic flow of $F_{\alpha,\beta}$ has exactly two geometrically distinct prime closed orbits if and only if the ratio $\alpha/\beta$ is irrational, since the flow is the round cogeodesic flow composed with a constant rotation.\footnote{Note that the Reeb flow on $\partial D^*(S^2,F_{\alpha,\beta}^*)$ is the cogeodesic flow of $F_{\alpha,\beta}$ and hence this is also clear from the moment map, see \cref{fig:ATF_Bdpq} (a), in analogy to the Reeb flow on irrational ellipsoids.}
    These two closed orbits project to the equator. 
    This is why these well-known Finsler metrics were originally introduced by Katok \cite{Ka74}.
    They have since then been studied from several different perspectives, including closed geodesics \cite{Zi83,Ra04}, flag curvature \cite{Sh02}, contact geometry \cite{HaPa08}, embedded contact homology (ECH) \cite{Fer24}, and Zermelo navigation \cite{BaRoSh04}.\footnote{For a general introduction to Finsler geometry see \cite{BaChSh00} and for some of the rich and interesting history see \cite{Ber03,Ch96} and \cite[Appendix B]{AbSaSch23}.}
\end{remark}

Denote by $\widetilde{\mu}_{S^2}$ the moment map on $T^*S^2\setminus 0_{S^2}$ that is obtained from $\mu_{S^2}$, which was defined in \cref{rmk:moment_map_S2}, by applying the integral affine transformation $M$ to its moment image as shown in \cref{fig:ATF_moment_map} (a), i.e.\ $\widetilde{\mu}_{S^2}=(H-J,H)$.
Then the shape of the moment image of $\widetilde{\mu}_{S^2}$ coincides with that of $\ATF_{2;1,1}$.

\begin{figure}[htb]
  \begin{center}   
    \begin{tikzpicture}[scale=0.8]
    
    \begin{scope}[shift={(-9.5,0)}]
        \node at (-0.8,1.7) {(a)};
        \fill[lightgray,opacity=0.75] 
            (2,2) -- (0,0) -- (2,-2);
        \draw[opacity=0.5,->]
            (0,-2) -- (0,2);
        \draw[opacity=0.5,->]
            (-0.5,0) -- (2,0);
        \draw[thick, mid arrow] 
            (0,0) -- node[above left, xshift=4pt, yshift=-4pt]{\tiny $\pmat{1 \\ 1}$} (2,2);
        \draw[thick, mid arrow]
            (0,0) --  node[below left, xshift=4pt, yshift=4pt]{\tiny $\pmat{1 \\ -1}$} (2,-2);
        \draw[fill=white, line width=1pt] (0,0) circle (2pt);
    \end{scope}

    \draw[->] (-7.25,0) to[out=-20,in=200] (-4.75,0);
    \node at (-6,-0.2) [below] {\tiny $M=\pmat{1 & -1 \\ 1 & 0}$};

    \begin{scope}[shift={(-4,0)}]
        \fill[lightgray,opacity=0.75] 
            (0,2) -- (0,0) -- (2,1) -- (2,2);
        \draw[opacity=0.5,->]
            (0,-2) -- (0,2);
        \draw[opacity=0.5,->]
            (-0.5,0) -- (2,0);
        \draw[thick, mid arrow] 
            (0,0) -- node[left, xshift=2pt]{\tiny $\pmat{0 \\ 1}$} (0,2);
        \draw[thick, mid arrow]
            (0,0) --  node[below right, xshift=-4pt, yshift=2pt]{\tiny $\pmat{2 \\ 1}$} (2,1);
        \draw[fill=white, line width=1pt] (0,0) circle (2pt);
    \end{scope}

    \begin{scope}[shift={(2,0)}]
        \node at (-1.4,1.7) {(b)};
        \fill[lightgray,opacity=0.75] 
            (0,1) -- (0,-1.5) -- (6,1.5);
        \draw[opacity=0.5,->]
            (0,-2) -- (0,2);
        \draw[opacity=0.5,->]
            (-0.5,-1.5) -- (6,-1.5);
        \draw[thick] 
            (0,-1.5) -- (0,1);
        \draw[thick]
            (0,-1.5) -- (6,1.5);
        \draw[fill=white, line width=1pt] (0,-1.5) circle (2pt);
        \draw[thick,dash pattern=on 7pt off 3pt] (0,1) --node[sloped, midway, above]{\tiny $\{(\beta+\alpha)H-(\beta-\alpha)J= 2\alpha\beta\}$} (6,1.5);
        \node at (6,1.5) [right] {\tiny $(2\alpha,\alpha)$};
        \node at (0,1) [left] {\tiny $(0,\beta)$};
        \node at (0,2) [left, opacity=0.5] {\tiny $H$};
        \node at (6,-1.5) [above, opacity=0.5] {\tiny $H-J$};
    \end{scope}
    \end{tikzpicture}
    \caption{(a) The moment image of $\mu_{S^2}$ and the moment image of $\widetilde{\mu}_{S^2}$. (b) The moment image of $E_{T^*S^2}(\alpha,\beta)$ under $\widetilde{\mu}_{S^2}$.}
    \label{fig:ATF_moment_map}
  \end{center}
\end{figure}
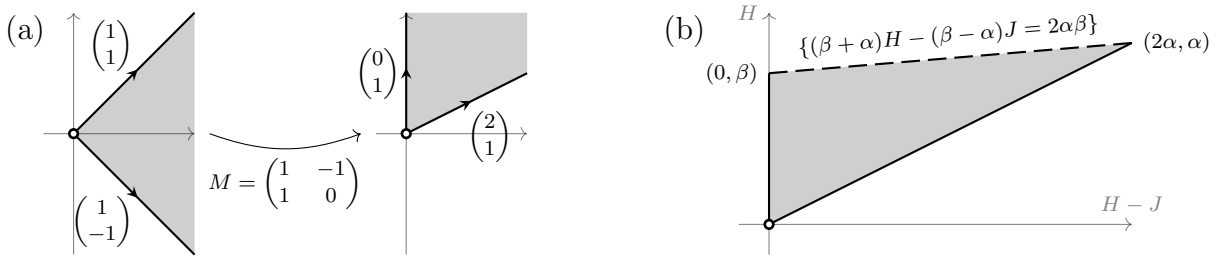

\begin{definition}
    Define $E_{T^*S^2}(\alpha,\beta):=\widetilde{\mu}_{S^2}^{-1}(\Delta_{2;1,1}(\alpha,\beta)) \cup 0_{S^2}$, where the triangle $\Delta_{2;1,1}(\alpha,\beta)$ was defined in \eqref{eq:def_Delta}.
    As before we abbreviate $B_{T^*S^2}(\lambda)=E_{T^*S^2}(\lambda,\lambda)$ for $\lambda \in \R_{>0}$.
\end{definition}

\begin{remark}\label{rmk:BTS2}
    By definition we have that $B_{T^*S^2}(\lambda)=D^*\left(S^2,\tau / \lambda \lvert\cdot\rvert_{g_0^*} \right)$, as shown in \cref{fig:ATF_moment_map} (b).
\end{remark}

\begin{proposition}\label{prop:ETS2_E211}
    $E_{T^*S^2}(\alpha,\beta)$ and $E_{2;1,1}(\alpha,\beta)$ are symplectomorphic.
\end{proposition}

\begin{proof}
    Pick $0 < \varepsilon_1 < \varepsilon_2 < \min(\alpha,\beta)$.
    Then, as shown in \cite[Chapter 2]{Ev23:book}, there exists a fibred symplectomorphism $$\psi:E_{T^*S^2}(\alpha,\beta) \setminus B_{T^*S^2}(\varepsilon_1) \to E_{2;1,1}(\alpha,\beta) \setminus B_{2;1,1}(\varepsilon_1).$$
    Our goal is now to show that $B_{T^*S^2}(\varepsilon_2)$ and $B_{2;1,1}(\varepsilon_2)$ are symplectomorphic via a symplectomorphism $\phi$ that agrees with $\psi$ near $\partial B_{T^*S^2}(\varepsilon_2)$, because then the symplectomorphism 
    $$
    \varphi: E_{T^*S^2}(\alpha,\beta) \to E_{2;1,1}(\alpha,\beta)
    $$
    obtained by gluing $\psi$ and $\phi$ proves the claim.
    Since $B_{T^*S^2}(\varepsilon_2)=D^*\left(S^2,\tau / \varepsilon_2 \lvert\cdot\rvert_{g_0^*} \right)$, as noted in \cref{rmk:BTS2}, we know that performing a symplectic cut along the boundary characteristics yields $\C P^1(\varepsilon_2) \times \C P^1(\varepsilon_2)$ and the thereby introduced symplectic divisor is an embedded symplectic $(+2)$-sphere $\Sigma_0$, as discussed in \cite{Aud07,HiPiWu16,OaUs16}.
    Similarly performing a symplectic cut on the boundary of $B_{2;1,1}(\varepsilon_2)$, yields $\C P^1(\varepsilon_2) \times \C P^1(\varepsilon_2)$ and introduces a $(+2)$-sphere $\Sigma_1$.
    Gromov's classical result in \cite[Section 2.4.E]{Gro85} shows that embedded symplectic $(+2)$-spheres in monotone $\C P^1 \times \C P^1$ are unique up to symplectic isotopies which immediately implies that they are unique up to Hamiltonian isotopy, because $H^1_\text{dR}(\C P^1 \times \C P^1;\R)$ vanishes.
    Therefore we find a symplectomorphism 
    $$
    \phi':(\C P^1(\varepsilon_2) \times \C P^1(\varepsilon_2),\Sigma_0) \to (\C P^1(\varepsilon_2) \times \C P^1(\varepsilon_2),\Sigma_1).
    $$
    Note that $\psi$ induces a symplectic embedding $\iota_\psi:\nu \Sigma_0 \setminus \Sigma_0 \hookrightarrow \nu \Sigma_1 \setminus \Sigma_1$, where $\nu \Sigma_i$ is a normal neighbourhood of $\Sigma_i$.
    Since both $\Sigma_0$ and $\Sigma_1$ are divisors that were introduced by a symplectic cut and $\psi$ is fibre preserving $\iota_\psi$ extends to $\Sigma_0$ and restricting to $\Sigma_0$ it yields a symplectomorphism $\iota_\psi|_{\Sigma_0}:\Sigma_0 \to \Sigma_1$.
    Then there exists a Hamiltonian diffeomorphism $\rho$ supported near $\Sigma_0$ such that $\phi:=\phi' \circ \rho$ is equal to $\iota_\psi$ in a possibly smaller neighbourhood $\nu \Sigma_0$.
    See \cite[Lemma 5.2.6]{ABEHS25} for details on how to construct this Hamiltonian diffeomorphism.
    In particular, this means that $\psi$ and $\phi$ glue to a symplectomorphism 
    $$
    \varphi: E_{T^*S^2}(\alpha,\beta) \to E_{2;1,1}(\alpha,\beta)
    $$
    as claimed.
\end{proof}

\begin{corollary}
\label{cor:211_ellipsoids}
    For $\alpha,\beta \in \R_{>0}$ the ellipsoid $E_{2;1,1}(\alpha,\beta)$ is symplectomorphic to the unit codisc bundle $D^*(S^2,F^*_{\alpha,\beta}):=\{F_{\alpha,\beta}^*\leq 1\}\subseteq T^*S^2$.
\end{corollary}

\begin{proof}
    Writing out the definition of $E_{T^*S^2}(\alpha,\beta)$ we find, as also shown in \cref{fig:ATF_moment_map},
    \begin{equation*}
        E_{T^*S^2}(\alpha,\beta)=\{(\beta+\alpha)H-(\beta-\alpha)J\leq 2\alpha\beta\}=\{F_{\alpha,\beta}^*\leq 1\}=D^*(S^2,F^*_{\alpha,\beta}).
    \end{equation*}
    Hence, \cref{prop:ETS2_E211} implies the claim.
\end{proof}

\begin{remark}
\label{rmk:volume_E211}
    In particular, \cref{cor:211_ellipsoids} proves that for $\lambda >0$ the $(2;1,1)$-ball $B_{2;1,1}(\lambda)$ is symplectomorphic to $D^*\left(S^2,\frac{\tau}{\lambda}\lvert\cdot\rvert_{g_0^*}\right)$, which means that 
    $$
    \vol(B_{2;1,1}(\lambda))=
    \vol\hspace{-2.4pt}\left(D^*\left(S^2,\tfrac{\tau}{\lambda}\lvert\cdot\rvert_{g_0^*}\right)\right)=
    \int_{D^*\left(S^2,\frac{\tau}{\lambda}\lvert\cdot\rvert_{g_0^*}\right)} \frac{\omega^2_\can}{2}= 
    \lambda^2.
    $$
    This formula holds more generally, i.e.\ for $\alpha,\beta \in \R_{>0}$ we have
    $$
    \vol(E_{2;1,1}(\alpha,\beta))=\text{Area}(\Delta_{2;1,1}(\alpha,\beta))=\alpha\beta,
    $$
    which follows from the usual facts about moment maps.
\end{remark}

\begin{remark}
    Another natural question is whether the Gromov width $c_{\mathrm{Gr}}$ of the ellipsoids $E_{2;1,1}(\alpha,\beta)$ and $E_{2,1}(\alpha,\beta)$ can be computed.
    In \cite{FerRam22}, Ferreira and Ramos compute the Gromov width in the round cases, i.e.\ for $B_{2;1,1}(\lambda)$ and $B_{2,1}(\lambda)$.
    This question will be addressed by Ramos, Schlenk and the second author in forthcoming work \cite{HaSch26}.
\end{remark}

It is well known that the interior of $D^*\left(S^2,\tau\lvert\cdot\rvert_{g_0^*}\right)$ compactifies to $\C P^1 (1) \times \C P^1(1)$.
See for example \cite{Aud07,OaUs16}.
Therefore, $E_{2;1,1}(\alpha,\beta)$ embeds into $\C P^1 (1) \times \C P^1(1)$ for small $\alpha,\beta$.
We will quantify for which $\alpha,\beta \in \R_{>0}$ this is the case, i.e.\ we will compute the set
\begin{equation}\label{eq:def_A211}
    \ca_{2;1,1}:=\left\{(\alpha,\beta) \in \R_{>0}^2 \,\middle|\, E_{2;1,1}(\alpha,\beta) \xhookrightarrow{\raisebox{-0.5ex}{\( \hspace{0.5em}\scriptstyle s \hspace{0.5em} \)}} \C P^1 
    (1) \times \C P^1(1)\right \}.
\end{equation}
Similarly, we define the corresponding set for the affine problem:
\begin{equation}\label{eq:def_E211}
    \ce_{2;1,1}:=\left\{(\alpha,\beta) \in \R_{>0}^2 \,\middle|\, E_{2;1,1}(\alpha,\beta)  \xhookrightarrow{\raisebox{-0.5ex}{\( \hspace{0.5em}\scriptstyle s \hspace{0.5em} \)}} B_{2;1,1}(1) \right \}.
\end{equation}

\begin{remark}
\label{rmk:Randers_interpretation_RP2}
    This discussion for the $(2;1,1)$-case generalises immediately to the $(2,1)$-case, because both $H$ and $J$, defined in \eqref{eq:H_and_J}, are invariant under the antipodal map.
    Only the period lattice has to be corrected, since the antipodal quotient divides the lengths of great circles by $2$.
    The analogue of \cref{cor:211_ellipsoids} gives that $E_{2,1}(\alpha,\beta)$ is symplectomorphic to $D^*(\R P^2, G^*_{\alpha,\beta}) \subseteq T^* \R P^2$, where $G_{\alpha,\beta}^*$ is defined analogously to $F_{\alpha,\beta}^*$, as in \eqref{eq:Randers_metric}, using the functions induced by $H/2$ and $J/2$ on the quotient.
    In particular, under the quotient map $G_{\alpha,\beta}^*$ pulls back to $1/2 F_{\alpha,\beta}^*=F_{2\alpha,2\beta}^*$.
    Since the interior of $B_{2,1}(1) \cong D^*(\R P^2, \tau/2 \lvert \cdot \rvert^*_{g_0})$ compactifies to $\C P^2(2)$, see again \cite{OaUs16,Aud07}, we define, in analogy to \eqref{eq:def_A211},
    \begin{equation}\label{eq:def_A21}
        \ca_{2,1}:=\left\{(\alpha,\beta) \in \R_{>0}^2 \,\middle|\, E_{2,1}(\alpha,\beta) \xhookrightarrow{\raisebox{-0.5ex}{\( \hspace{0.5em}\scriptstyle s \hspace{0.5em} \)}} \C P^2(2)\right\}.
    \end{equation}
\end{remark}

\section{\texorpdfstring{$\ca_{2;1,1}$}{A211} has a symplectic staircase structure}
\label{sec:A211_staircase}

In this section we repeat the main arguments in \cite{ABEHS25} for the case $X := \C P^1(1) \times \C P^1(1)$.
We do this in order to show that the methods used there for the target $\C P^2$ extend to this case.
However, we will not repeat the definitions and the whole setup of \cite{ABEHS25}, as they are verbatim the same here.
We only formulate a slight generalisation of the results in \cite{ABEHS25} and we will comment on the generalisation at the appropriate point.

Assume in the following that $\iota: E_{2;1,1}(\alpha,\beta) \hookrightarrow X$ is a symplectic embedding for some positive real numbers $\alpha,\beta \in \R_{>0}$.
By definition this means that there is a symplectic embedding 
$$
\hat{\iota}:E_{2;1,1}(\alpha+\varepsilon,\beta+\varepsilon) \hookrightarrow X
$$
that extends $\iota$, i.e.\ $\hat{\iota}|_{E_{2;1,1}(\alpha,\beta)}=\iota$.
Denote $U:=\iota(E_{2;1,1}(\alpha,\beta))$, $V:=X\setminus U$ and pick a Delzant pavilion $(\bm{\rho},\bm{\lambda})$ that is completely contained in $\Delta_{2;1,1}(\alpha+\varepsilon,\beta+\varepsilon)\setminus \Delta_{2;1,1}(\alpha,\beta)$, as shown in \cref{fig:pavillion_choice}.
Moreover, denote the pavilion blow-up of $X$ along $\hat{\iota}$ by $\widetilde{X}=\text{Pav}_{\bm{\rho},\bm{\lambda}}^{\hat{\iota}}(X)$, equipped with a symplectic form $\widetilde{\omega}$ and a compatible almost complex structure $\widetilde{J}$, that makes the exceptional curves $C_1,\ldots,C_m$ of the pavilion $\widetilde{J}$-holomorphic.

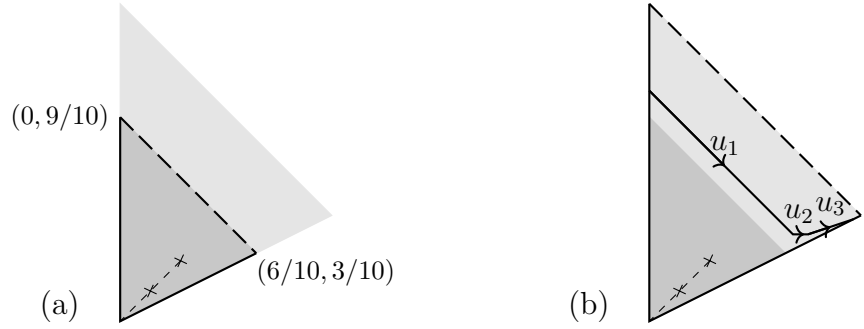
\begin{figure}[htb]
  \begin{center}   
    \begin{tikzpicture}
    \begin{scope}[shift={(-6,0)}]
        \node at (-0.8,0.2) {(a)};
        \filldraw[lightgray,opacity=0.75] (0,2.7) -- (0,0) -- (1.8,0.9) -- cycle;
        \filldraw[lightgray,opacity=0.4] (0,4.2) -- (0,0) -- (2.8,1.4) -- cycle;
        \draw[thick] (0,2.7) -- (0,0) -- (1.8,0.9);
        \node at (0,2.7) [left] {\footnotesize $(0,9/10)$};
        \node at (1.8,0.9) [below right, xshift=-4pt, yshift=2pt] {\footnotesize $(6/10,3/10)$};
        \draw[dashed] (0.8,0.8) node[cross] {} -- (0.4,0.4) node[cross] {} --(0,0);
        \draw[thick,dash pattern=on 7pt off 3pt] (0,2.7) -- (1.8,0.9);
    \end{scope}

    \begin{scope}[shift={(1,0)}]
        \node at (-0.8,0.2) {(b)};
        \filldraw[lightgray,opacity=0.75] (0,2.7) -- (0,0) -- (1.8,0.9) -- cycle;
        \filldraw[lightgray,opacity=0.4] (0,4.2) -- (0,0) -- (2.8,1.4) -- cycle;
        \draw[thick] (0,4.2) -- (0,0) -- (2.8,1.4);
        \draw[dashed] (0.8,0.8) node[cross] {} -- (0.4,0.4) node[cross] {} --(0,0);
        \draw[thick] (0,3.05) -- (1.9,1.15) -- (2.1,1.15) -- (2.7,1.35);
        \draw[->,thick] (0,3.05) -- (1,2.05) node [above] {\(u_1\)};
        \draw[->,thick] (2,1.15) -- (2.05,1.15) node [above, xshift=-2pt] {\(u_2\)};
        \draw[thick,dash pattern=on 7pt off 3pt] (0,4.2) -- (2.8,1.4);
        \draw[->,thick] (2.1,1.15) -- (2.4,1.25) node [above] {\(u_3\)};
    \end{scope}
    
    \end{tikzpicture}
    \caption{(a) The almost toric base diagram $\ATF_{2;1,1}(3/10,9/10)$ of $E_{2;1,1}(3/10,9/10)$. (b) A choice of pavilion $(\bm{\rho},\bm{\lambda})$ for $\Delta_{2;1,1}(3/10+\varepsilon,9/10+\varepsilon)$.}
    \label{fig:pavillion_choice}
  \end{center}
\end{figure}

Recall that the resolution of the $A_1$-surface singularity defined by the pavilion $(\bm{\rho},\bm{\lambda})$ covers the minimal resolution of the $A_1$-surface singularity and therefore there is a distinguished index $i_\sect$ such that $C_{i_\sect}$ is the proper transform of the unique sphere in the minimal resolution.

To make the exposition more uniform we assume that there always exist $C_{i_\sect-1}$ and $C_{i_\sect+1}$, i.e.\ exceptional curves to the "left" and the "right" of $C_{i_\sect}$, see again \cref{fig:pavillion_choice}. 
In this setup we then have the following theorem.

\begin{theorem}[\normalfont{\cite[Theorem 5.1.1]{ABEHS25}}]
\label{thm:existence_ruling}
    Assuming that $\widetilde{J}$ is chosen generically on $V$ there exists an embedded square zero $\widetilde{J}$-holomorphic rational curve $\widetilde{C}\subseteq \widetilde{X}$ that intersects $C_{i_\sect}$ once and is disjoint from all the other exceptional curves of the pavilion.
\end{theorem}

\begin{remark}
    We will prove a slightly stronger version of the theorem than what is proven in \cite[Theorem 5.1.1]{ABEHS25}: in our version the pavilion is not assumed to be minimal.
    Recall that a minimal pavilion is a pavilion that corresponds to the minimal resolution. In our situation a minimal pavilion has a single entry.
\end{remark}

Our next theorem concerns the structure of the regulation defined by the ruling in \cref{thm:existence_ruling}.

\begin{theorem}[\normalfont{\cite[Theorem 5.1.4]{ABEHS25}}]
\label{thm:structure_regulation}
    Assuming that $\widetilde{J}$ is chosen generically on $V$, the curve $\widetilde{C}$ from \cref{thm:existence_ruling} defines a $\widetilde{J}$-holomorphic regulation of $\widetilde{X}$ such that $C_{i_\sect}$ is a section and such that there are two broken rulings:
    $$\cf_1=C_1\cup  \ldots \cup C_{i_\sect-1} \cup E_1 \quad\text{and}\quad \cf_2=C_{i_\sect+1}\cup \ldots \cup C_m \cup E_2$$
    where $E_1$ and $E_2$ are exceptional $(-1)$-curves which intersect $\cc=C_1 \cup \ldots \cup C_m$ in a unique point.
\end{theorem}

We give the proofs of \cref{thm:existence_ruling} and \cref{thm:structure_regulation} at the end of this section.
Define the configuration of symplectic spheres $\ct:=\cf_1 \cup C_{i_\sect} \cup \cf_2$.
With these theorems in place results of~\cite{ABEHS25} now imply that $(2;1,1)$-ellipsoid embeddings in $X$ are unique up to symplectomorphism. 

\begin{theorem}
    Given two symplectic embeddings $\iota,\iota':E_{2;1,1}(\alpha,\beta) \hookrightarrow X$ there exists a symplectomorphism $\Phi \in \symp(X)$ that intertwines $\iota$ and $\iota'$.
\end{theorem}

\begin{proof}
    By definition of symplectic embedding there exist extensions to symplectic embeddings $\hat{\iota},\hat{\iota}':E_{2;1,1}(\alpha+\varepsilon,\beta+\varepsilon) \hookrightarrow X$. 
    Choose a pavilion $(\bm{\rho},\bm{\lambda})$ as shown in \cref{fig:pavillion_choice} and denote the respective pavilion blow-ups by $((\widetilde{X},\widetilde{\omega}),\ct)$ and $((\widetilde{X}',\widetilde{\omega}'),\ct')$, where $\ct$ and $\ct'$ are the configurations of symplectic spheres, as in \cref{thm:structure_regulation}.
    Then, exactly as proven in \cite[Corollary 5.2.2]{ABEHS25}, there exists a symplectomorphism 
    $$
    \Psi:\big((\widetilde{X},\widetilde{\omega}),\ct\big) \to \big((\widetilde{X}',\widetilde{\omega}'),\ct'\big)
    $$
    and, in analogy to \cite[Corollary 5.2.5]{ABEHS25}, this symplectomorphism $\Psi$ glues to a symplectomorphism $\Phi \in \text{Symp}(X)$ such that $\iota'=\Phi \circ \iota$.\footnote{See also the end of the proof of \cref{prop:ETS2_E211} where a similar gluing argument is used.}
\end{proof}

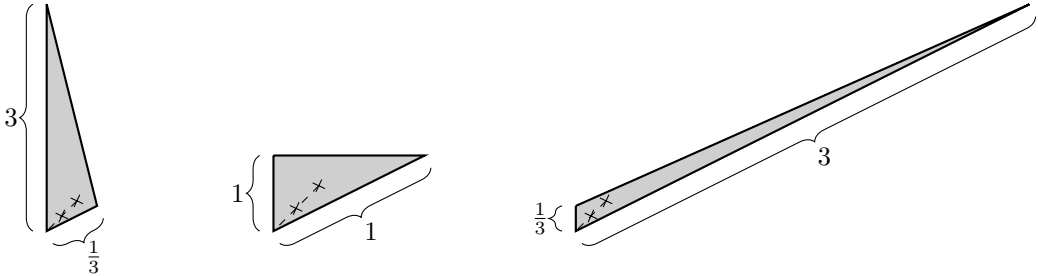
\begin{figure}[htb]
  \begin{center}   
    \begin{tikzpicture}
    \begin{scope}[shift={(0,0)}]
        \filldraw[lightgray,opacity=0.75] (0,1) -- (0,0) -- (2,1) -- cycle;
        \draw[thick] (0,1) -- (0,0) -- (2,1) -- (0,1);
        \draw[dashed] (0.6,0.6) node[cross] {} -- (0.3,0.3) node[cross] {} --(0,0);
        \draw [decorate,decoration={brace,amplitude=5pt,raise=1ex}] (0,0) -- (0,1) node[midway,xshift=-2.6ex,yshift=0ex]{\footnotesize $1$};
        \draw [decorate,decoration={brace,amplitude=5pt,raise=1ex}] (2,1) -- (0,0) node[midway,xshift=1.6ex,yshift=-2.8ex]{\footnotesize $1$};
    \end{scope}

    \begin{scope}[shift={(4,0)}]
        \filldraw[lightgray,opacity=0.75] (0,1/3) -- (0,0) -- (6,3) -- cycle;
        \draw[thick] (0,1/3) -- (0,0) -- (6,3) -- (0,1/3);
        \draw[dashed] (0.4,0.4) node[cross] {} -- (0.2,0.2) node[cross] {} --(0,0);
        \draw [decorate,decoration={brace,amplitude=5pt,raise=1ex}] (0,0) -- (0,1/3) node[midway,xshift=-2.6ex,yshift=0ex]{\footnotesize $\frac{1}{3}$};
        \draw [decorate,decoration={brace,amplitude=5pt,raise=1ex}] (6,3) -- (0,0) node[midway,xshift=1.6ex,yshift=-2.8ex]{\footnotesize $3$};
    \end{scope}

    \begin{scope}[shift={(-3,0)}]
        \filldraw[lightgray,opacity=0.75] (0,3) -- (0,0) -- (2/3,1/3) -- cycle;
        \draw[thick] (0,3) -- (0,0) -- (2/3,1/3) -- (0,3);
        \draw[dashed] (0.4,0.4) node[cross] {} -- (0.2,0.2) node[cross] {} --(0,0);
        \draw [decorate,decoration={brace,amplitude=5pt,raise=1ex}] (0,0) -- (0,3) node[midway,xshift=-2.6ex,yshift=0ex]{\footnotesize $3$};
        \draw [decorate,decoration={brace,amplitude=5pt,raise=1ex}]  (2/3,1/3) -- (0,0) node[midway,xshift=1.6ex,yshift=-2.8ex]{\footnotesize $\frac{1}{3}$};
    \end{scope}
    
    \end{tikzpicture}
    \caption{In the middle an almost toric base diagram for $X$, which yields embeddings of $E_{2;1,1}(1-\varepsilon,1-\varepsilon)$ into $X$ for all $0<\varepsilon<1$. On the left/right the almost toric base diagram obtained from the middle one by a mutation of the top left/right vertex.}
    \label{fig:first_mutations}
  \end{center}
\end{figure}

\begin{corollary}
\label{cor:obstruction_211_stair}
    For all $i \in \Z$ no symplectic embedding $E_{2;1,1}(\alpha,\beta) \hookrightarrow X$ exists if both $\alpha > \frac{m_{i+1}}{m_i}$ and $\beta > \frac{m_{i+1}}{m_{i+2}}$, where $\{m_i\}_{i \in \Z}$ is the sequence defined in \eqref{eq:211_seq}.
\end{corollary}

\begin{proof}
    This follows exactly by the same methods as \cite[Theorem 5.3.1]{ABEHS25}.
    The only difference is that we are only interested in the easiest branch of the weighted Markov tree, and hence the numerology is much easier. 
    See \cref{fig:first_mutations} for the first two relevant mutations.
    The obstruction comes from visible curves in the spirit of \cite{McDSie25}.
\end{proof}

\begin{corollary}
\label{cor:staircase_211}
    We have $\ca_{2;1,1} \cap (0,2+\sqrt{3})^2=\Stair(2;1,1)$.
\end{corollary}

\begin{proof}
    The obstructive part follows from \cref{cor:obstruction_211_stair}, whereas the constructive part follows from the usual almost toric methods, in particular visible embeddings.
    See for example \cite[Subsection 1.3]{ABEHS25}.
\end{proof}

Having proven the main results of this section let us come to the proofs of \cref{thm:existence_ruling} and \cref{thm:structure_regulation}. 
It will be important for us to prove both theorems with non-minimal choices of pavilions.
Therefore, we first investigate how the intersection matrix $M_\cc:=(C_i \cdot C_j)_{ij}$ and its inverse $(M_\cc)^{-1}$ change under blow-downs/ups.

\begin{lemma}
\label{lma:intmatrix_blow}
    Assume that $M_\cd$ is the intersection matrix of the chain $\cd=(D_1,\ldots,D_m)$ and $\cc$ is the chain obtained from $\cd$ by blowing up the node $D_r \cap D_{r+1}$, i.e.\ 
    $$
    \cc:=(D_1,\ldots,D_r-E,E,D_{r+1}-E,D_{r+2},\ldots,D_m)=(C_1,\ldots,C_{m+1}),
    $$
    then the entries of the inverse of the intersection matrix $M_{\cc}$ can be computed from $(M_{\cd})^{-1}$ via:
    $$
    (M_\cc)^{-1}_{ij}=\bm{u}_i^T (M_{\cd})^{-1} \bm{u}_j - \delta_{i,r+1}\delta_{r+1,j},
    $$
    where $\delta_{i,j}$ is the Kronecker delta and
    \begin{equation}
    \label{eq:U_collapsing_vectors}
    \bm{u}_i=
        \begin{cases}
            \bm{e}'_i & i\leq r \\
            \bm{e}'_r + \bm{e}'_{r+1}  & i=r+1\\
            \bm{e}'_{i-1} & i\geq r+2 
        \end{cases}
    .
    \end{equation}
    Here $\bm{e}'_i$ denotes the $i$th unit vector of length $m$.
    In matrix form this means $$(M_\cc)^{-1}=U^T(M_\cd)^{-1}U-\bm{e}_{r+1}\bm{e}_{r+1}^T,$$
    where $U=(\bm{u}_1 \, \cdots \, \bm{u}_{m+1})$ and $\bm{e}_i$ denotes the $i$th unit vector of length $m+1$.
\end{lemma}

\begin{proof}
    The formula is just the base change from $(D_1,\ldots,D_m,E)$ to $(C_1,\ldots,C_{m+1})$.
\end{proof}

\begin{remark}
    In essence \cref{lma:intmatrix_blow} says that $(M_\cc)^{-1}$ can be computed from $(M_\cd)^{-1}$ by inserting a row/column that is the sum of the two adjacent rows/columns and the diagonal entry is the sum of the adjacent $2 \times 2$ block minus $1$.
\end{remark}

Assume for the moment that $\cc$ is obtained from $\cd$ via a single blow-up, i.e.\ that $\cc$ is as in the setup of \cref{lma:intmatrix_blow}.
We denote by $\widetilde{C}$ the holomorphic curve in $\widetilde{X}$ that is obtained by stretching curves in the class $A:=[\C P^1 \times \{\text{pt}\}] \in H_2(X;\Z)$, as explained in \cite[Section 5.4]{ABEHS25}.
Set $B:=[\{\text{pt}\} \times  \C P^1]$.
We then write
\begin{equation}
\label{eq:expressions_C}
    \widetilde{C}=d_0 (A+B) + \sum_{j=1}^{m+1} d_j C_j \qquad\text{and}\qquad K_\cc = -2 (A+B) + \sum_{j=1}^{m+1} l_j C_j.
\end{equation}
Note that the class $A+B$ can be represented in $V$.
Define $\xi_i=\widetilde{C}\cdot C_i$, which means $\bm{\xi}=M_\cc \bm{d}$, where $\bm{d}=(d_1, \ldots, d_{m+1})^T$.
Also define $\mu:=\xi_{r+1}$, the intersection multiplicity of $\widetilde{C}$ with the exceptional sphere $E$.
It is not hard to see that 
$$
\bm{l}=(k_1,\ldots,k_r,k_r+k_{r+1}+1,k_{r+1}, \ldots,k_m)^T,
$$
where $\bm{k}$ is the discrepancy vector of $\cd$.
Define 
\begin{equation}
\label{eq:formal_int}
    \bm{\chi}:=(\xi_1,\ldots,\xi_{r-1},\xi_r+\mu,\xi_{r+2}+\mu,\xi_{r+3},\ldots,\xi_{m+1})^T,
\end{equation}
the vector of formal intersection multiplicities of the curve $\widetilde{C}$ with $\cd$ after blowing down $E$.
In particular, we have $U\bm{\xi}=\bm{\chi}$, where $U$ is the matrix introduced in \cref{lma:intmatrix_blow}.

\begin{lemma}
\label{lma:invar_expression}
    In this setup we have
    $$
    \bm{\xi}^T (M_\cc)^{-1} \bm{\xi} + \bm{l}^T\bm{\xi} = \bm{\chi}^T (M_\cd)^{-1} \bm{\chi} + \bm{k}^T\bm{\chi} - \mu(\mu-1).
    $$
\end{lemma}

\begin{proof}
    This is just a computation, which we do termwise.
    By \cref{lma:intmatrix_blow} we have 
    $$
    \bm{\xi}^T (M_\cc)^{-1} \bm{\xi}
    =\bm{\xi}^T(U^T(M_\cd)^{-1}U-\bm{e}_{r+1}\bm{e}_{r+1}^T)\bm{\xi}
    =(U\bm{\xi})^T(M_\cd)^{-1}(U\bm{\xi})-\xi_{r+1}^2
    =\bm{\chi}^T(M_\cd)^{-1}\bm{\chi}-\mu^2
    $$
    and
    $$
    \bm{l}^T\bm{\xi}=\sum_{i=1}^{r-1}k_i\xi_i+k_r(\xi_r+\mu)
    +k_{r+1}(\xi_{r+2}+\mu)+\sum_{i=r+2}^m k_i \xi_{i+1}+\mu=\bm{k}^T\bm{\chi}+\mu,
    $$
    which completes the proof.
\end{proof}

\begin{remark}
    We chose to focus on the case of node blow-ups after \cref{lma:intmatrix_blow}, since the case of blowing up a smooth point on one of the chain components is easier.
    Importantly, the formula given in \cref{lma:invar_expression} continues to hold.
    For a blow-up at a smooth point of $D_r$, the new exceptional column, in the notation of \cref{lma:intmatrix_blow}, is $\bm u_{r+1}=\bm e'_r$, its discrepancy is $k_r+1$, and the formal blow-down replaces $\xi_r$ by $\xi_r+\mu$ while leaving all other surviving entries unchanged.
\end{remark}

Recall that in the general case $\cc$, the chain of spheres that is defined by the choice of pavilion, is obtained from the minimal resolution of the $A_1$-surface singularity, i.e.\ the resolution that introduces a single exceptional $(-2)$-curve $\Sigma_-$, by a sequence of blow-ups.
We denote the formal intersection multiplicity of $\widetilde{C}$ with $\Sigma_-$ by $\chi \in \N$.
This number is obtained by iterating \eqref{eq:formal_int}.
From the expressions in \eqref{eq:expressions_C} and the computations in the proof of \cref{lma:invar_expression} we obtain
\begin{equation}
\label{eq:Csquare_KC}
    \widetilde{C}^2=2d_0^2 + \bm{\xi}^T (M_\cc)^{-1} \bm{\xi} = 2d_0^2 -\frac{\chi^2}{2} - \sum\nolimits_\nu \mu_\nu^2; \quad K_\cc \cdot \widetilde{C} = -4 d_0 + \bm{l}^T \bm{\xi} = -4d_0 + \sum\nolimits_\nu \mu_\nu,
\end{equation}
where the $\mu_\nu \in \N$ are iteratively defined through the sequence of blow-downs, as in the discussion after equation \eqref{eq:expressions_C}.
Recall that the $A_1$-surface singularity is crepant and so its discrepancy vanishes.
Hence, the adjunction formula for $\widetilde{C}$ reads
\begin{equation}
\label{eq:adj_C}
    \sum_{x \in \text{Sing}(\widetilde{C})} \delta_x = \frac{\widetilde{C}^2 + K_\cc \cdot \widetilde{C}}{2}+1
    = d_0^2 - 2d_0 - \frac{1}{2}\left(\frac{\chi^2}{2} + \sum\nolimits_\nu \mu_\nu(\mu_\nu-1)\right)+1,
\end{equation}
using \eqref{eq:Csquare_KC}.
Because $0 < d_0 \leq 1/2$, as can be shown completely analogous to \cite[Lemma 5.4.2]{ABEHS25}, we have that $-3/4\leq d_0^2 - 2d_0<0$ and therefore $\text{Sing}(\widetilde{C}) = \emptyset$, i.e.\ $\widetilde{C}$ is embedded, because the term in the bracket on the right hand side of \eqref{eq:adj_C} is non-negative and the $\delta_x$ are non-negative integers.
This argument is completely analogous to \cite[Proposition~5.4.3]{ABEHS25}, and arguing as in \cite[Proposition~5.4.5]{ABEHS25} we can assume that $\widetilde{C}$ is a square $0$ curve.
In particular, we have that the adjunction formula \eqref{eq:adj_C} for $\widetilde{C}$ becomes 
\begin{equation}
\label{eq:adj_C_updated}
    1 = 2d_0 - d_0^2 + \frac{\chi^2}{4} + \frac{1}{2}\sum\nolimits_\nu \mu_\nu(\mu_\nu-1).
\end{equation}

\begin{lemma}
    We have $\chi=1$.
\end{lemma}

\begin{proof}
    If $\chi=0$, then the adjunction formula \eqref{eq:adj_C_updated} together with $2d_0-d_0^2 \in (0,3/4]$ gives
    $$
    \frac12\sum_\nu\mu_\nu(\mu_\nu-1)
    =1-2d_0+d_0^2\in[1/4,1),
    $$
    contradicting the integrality of the left-hand side.
    If $\chi \geq 2$ then
    $$
    1=2d_0-d_0^2 + \frac{\chi^2}{4} + \frac{1}{2}\sum\nolimits_\nu \mu_\nu(\mu_\nu-1) \geq 2d_0-d_0^2 + 1 > 1,
    $$
    which is also a contradiction and hence we conclude $\chi=1$.
\end{proof}

\begin{corollary}
    We have $\bm{\xi}=\bm{e}_{i_\sect}$, i.e.\ $\widetilde{C}$ intersects $\cc$ exactly once and the intersection is with the proper transform of $\Sigma_-$, which is $C_{i_\sect}$.
\end{corollary}

\begin{proof}
    Since $\widetilde{C}$ is a square zero curve and $\chi=1$ we have by \eqref{eq:Csquare_KC} that
    $$
    0=2d_0^2 - \frac{1}{2} - \sum\nolimits_\nu \mu_\nu^2.
    $$
    As $0<d_0\leq 1/2$, both terms
    $2d_0^2-1/2$ and $-\sum_\nu\mu_\nu^2$ are nonpositive.
    This implies $d_0=1/2$ and $\mu_\nu=0$ for every $\nu$.
    Since every $\mu_\nu$ vanishes, the formal blow-down rule
    at each step simply deletes the zero entry corresponding
    to the contracted component and leaves all other entries
    unchanged. After all blow-downs, only the entry corresponding
    to $C_{i_\sect}$ remains, and its value is $\chi=1$.
    Hence $\xi_{i_\sect}=1$ and $\xi_j=0$ for every
    $j\neq i_\sect$.
\end{proof}

This proves \cref{thm:existence_ruling}, and \cref{thm:structure_regulation} is an immediate consequence: the two "tails" of $$\cc=(C_1,\ldots,C_{i_\sect-1},C_{i_\sect},C_{i_\sect+1}, \ldots, C_m),
$$
i.e.\ $(C_1,\ldots,C_{i_\sect-1})$ and $(C_{i_\sect+1}, \ldots, C_m)$, have to appear as irreducible components of a broken ruling.
Moreover, both of the tails are derived from a single exceptional $(-1)$-sphere.
Therefore, the existence of the exceptional $(-1)$-spheres, attached to a unique sphere of the tails, is forced.

\section{Proving \texorpdfstring{$\ca_{2;1,1}=\ce_{2;1,1}$}{A211=E211}}
\label{sec:A211=E211}

As in the previous section, denote $X := \C P^1(1) \times \C P^1(1)$, $U:=\iota(E_{2;1,1}(\alpha,\beta))$ and $V:=X\setminus U$.
The equality $\ca_{2;1,1}=\ce_{2;1,1}$ is a consequence of the following theorem, whose proof is the main content of this section.

\begin{theorem}
\label{thm:quadric_complement}
    Assume that $\iota:E_{2;1,1}(\alpha,\beta) \hookrightarrow X$ is a symplectic embedding. 
    Then there exists an embedded symplectic $(+2)$-sphere $S$ in $V$.
\end{theorem}

Before proving \cref{thm:quadric_complement} let us formulate the main corollary and prove it.

\begin{corollary}
    We have that $\ca_{2;1,1}= \ce_{2;1,1}$.
\end{corollary}

\begin{proof}
    Assume that $\iota:E_{2;1,1}(\alpha,\beta) \hookrightarrow X$ is a symplectic embedding, i.e.\ that $(\alpha,\beta) \in \ca_{2;1,1}$.
    By \cref{thm:quadric_complement} there exists an embedded symplectic $(+2)$-sphere $S \subset V$.
    It is well known that such spheres are unique up to Hamiltonian isotopy. 
    See for example \cite[Section 2.4.E]{Gro85}, as in the proof of \cref{prop:ETS2_E211}.
    Therefore, there exists a Hamiltonian diffeomorphism $\Phi \in \text{Ham}(X)$ such that $\Phi(S)=\Delta$, where $\Delta$ is the diagonal sphere.
    $X\setminus\Delta$ is symplectomorphic to $\Interior\left(D^*\left(S^2,\tau \lvert\cdot \rvert_{g_0^*}\right)\right)$, as shown, for example, in \cite{Aud07,OaUs16}, which, in turn, is symplectomorphic to $\Interior(B_{2;1,1}(1))$ by \cref{cor:211_ellipsoids}.
    This means that $(\alpha,\beta) \in \ce_{2;1,1}$, i.e.\ $\ca_{2;1,1}\subseteq \ce_{2;1,1}$.
    
    The opposite inclusion $\ca_{2;1,1}\supseteq \ce_{2;1,1}$ follows by a scaling argument, where scaling means scaling by the Liouville flow.
    Assume that $E_{2;1,1}(\alpha,\beta) \hookrightarrow B_{2;1,1}(1)$.
    Then there exists a small $\varepsilon >0$ such that $E_{2;1,1}(\alpha+\varepsilon,\beta+\varepsilon) \hookrightarrow B_{2;1,1}(1)$.
    Hence, for $0<\lambda<1$ we obtain an embedding 
    $$E_{2;1,1}(\lambda(\alpha+\varepsilon),\lambda(\beta+\varepsilon)) \hookrightarrow B_{2;1,1}(\lambda) \subseteq B_{2;1,1}(1)\setminus \partial B_{2;1,1}(1)$$
    and after compactifying this means in particular $E_{2;1,1}(\alpha,\beta) \hookrightarrow X$.
    This finishes the proof.
\end{proof}

For proving \cref{thm:quadric_complement} we have to construct a symplectic divisor in the complement of the embedded $(2;1,1)$-ellipsoid $U$.
If the configuration $\ct$ provided by \cref{thm:structure_regulation} consisting of the two broken rulings and the proper transform of the section is chain-shaped, then the $(+2)$-sphere in the complement can easily be constructed by toric methods, as shown in \cref{fig:quadric_complement}.
If not, we need other techniques to construct the $(+2)$-sphere.

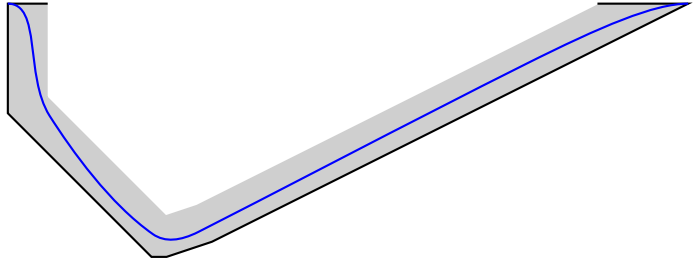
\begin{figure}[htb]
  \begin{center}   
    \begin{tikzpicture}
        \begin{scope}
            \clip (0,4.5) -- (0,3.05) -- (1.9,1.15) -- (2.1,1.15) -- (2.7,1.35) -- (9,4.5) -- (0,4.5);
            \draw[lightgray, opacity=0.75, line width=30pt, line cap=round]
                (0,4.5) -- (0,3.05) -- (1.9,1.15) -- (2.1,1.15) -- (2.7,1.35) -- (9,4.5);
        \end{scope}
        \draw[thick] (15pt,4.5) -- (0,4.5) -- (0,3.05) -- (1.9,1.15) -- (2.1,1.15) -- (2.7,1.35) -- (9,4.5) -- (7.8,4.5);
        \draw[thick, blue]
            (0,4.5)
            .. controls (0.45,4.50) and (0.22,3.55) .. (0.55,3.02)
            .. controls (0.85,2.55) and (1.35,1.85) .. (1.92,1.45)
            .. controls (2.08,1.34) and (2.28,1.36) .. (2.52,1.48)
            .. controls (3.35,1.90) and (5.35,2.95) .. (6.85,3.70)
            .. controls (7.85,4.20) and (8.55,4.50) .. (9,4.50);
    \end{tikzpicture}
    \caption{A toric fibration on a normal neighbourhood $\co(\ct)$ of the chain of symplectic spheres $\ct=(E_1,C_1,C_\sect,C_3,E_2)$. In blue an embedded symplectic $(+2)$-sphere in the neighbourhood $\co(\ct)$, as a visible surface. The figure uses the pavilion shown in \cref{fig:pavillion_choice}.}
    \label{fig:quadric_complement}
  \end{center}
\end{figure}

The following theorem is due to Mu{\~n}oz and Rojo.

\begin{theorem}[\normalfont{\cite[Theorem 5.4]{MuRo25}}]
\label{thm:neigh_divisor_smoothing}
    Let $(X,\omega)$ be a symplectic $4$-manifold, and suppose that $C_j \subseteq (X,\omega)$ are symplectic surfaces that intersect transversely and positively.
    Assume that $D=\sum_j a_j C_j$ is a symplectic divisor, i.e.\ a finite formal sum with $a_j \in \N$, that satisfies $a_j > 0$ and $[D]\cdot [C_j] \geq 0$ for all $j$.
    
    Define $\cc:=\cup_j C_j$.
    Then, given a tubular neighbourhood $\nu(\cc)$, there exists a symplectic surface $\Sigma \subseteq \nu(\cc)$ that represents $[D]$.
    Moreover, $\Sigma$ intersects each $C_j$ transversely and positively, which means in particular that $\Sigma$ is disjoint from $C_j$ if $[D]\cdot[C_j]=0$.
\end{theorem}

\begin{remark}
    \cref{thm:neigh_divisor_smoothing} can be viewed as a local symplectic analogue of Bertini's theorem in algebraic geometry. 
\end{remark}

\begin{remark}
\label{rmk:neigh_divisor_smoothing_connectedness}
    \cref{thm:neigh_divisor_smoothing} does not address the connectedness of the symplectic hypersurface that is constructed.
    However, investigating the proof of \cref{thm:neigh_divisor_smoothing} it is easy to deduce that it constructs a connected symplectic hypersurface, if $\cc$ is connected and $[D]^2>0$.
    
    The reason is the following: if $[D]^2>0$, then
    $$
    [D]^2 = [D] \cdot \Big(\sum\nolimits_j a_j [C_j] \Big) = \sum\nolimits_j a_j ([D]\cdot C_j),
    $$
    which means that $n_j:=[D]\cdot C_j >0$ for at least one $j$.
    Mu{\~n}oz and Rojo construct $\Sigma$ by first localising $X:=\nu(\cc)$ and then taking the canonical section $\sigma$ of the line bundle $L:=\co_X(D)$, whose zero set is the divisor $D$ and then perturbing it to $\sigma-\varepsilon s$.
    The section $s$ is chosen so that $s|_{C_j}$ vanishes at $n_j$ points on $C_j$ and such that $s$ is constant near the nodes.
    Near a zero of $s|_{C_j}$ the local perturbation looks like $w^{a_j}=\varepsilon z$ in local coordinates, where $z$ is the base coordinate and $w$ is the vertical coordinate in the tubular neighbourhood $X$.
    This connects the $a_j$ local sheets over $C_j$.
    Hence, for at least one component $C_j$ the multiple local sheets are connected together. 
    A similar argument then yields that over a node $C_i \cap C_k$ the sheets over $C_i$ are connected with those over $C_k$, and because $\cc$ is connected the connectedness property propagates from $C_j$.
\end{remark}

\begin{lemma}\label{lma:num_criterion_under_blow}
    Assume that $\cc = \cup_j C_j \subseteq (X,\omega)$ is a symplectic normal crossing divisor as in \cref{thm:neigh_divisor_smoothing} and that $D=\sum_j a_j C_j$ satisfies the numerical criterion of \cref{thm:neigh_divisor_smoothing}.
    If we assume that $\cc'=\cup_j C'_j$ is obtained from $\cc$ by blowing up a smooth point or a node, then there is a divisor $D'=\sum_j a'_j C_j'$ that still satisfies the numerical criterion of \cref{thm:neigh_divisor_smoothing} and has the same numerical invariants, i.e.\ Chern pairing and self-intersection number, as $D$.
\end{lemma}

\begin{proof}
    This follows from the usual facts about blow-ups.
    Denote the blow-up by $\pi:\widetilde{X} \to X$, introducing the exceptional sphere $E$.
    Then we have 
    $$
    H_2(\widetilde{X};\Z) \cong \pi^* H_2(X;\Z) \oplus \Z[E],
    $$
    where $\pi^*A \cdot \pi^*B= A \cdot B$, $\pi^*A \cdot E= 0$ and $E^2=-1$ are the intersection identities for the total transform.
    Then $D':=\pi^*D$ is the divisor that satisfies the claim.
    
    Let us do the calculation by hand for blowing up a smooth point, i.e.\ we assume that $\cc'$ is obtained from $\cc$ by blowing up a smooth point on $C_i$. 
    Then 
    $$
    D'= \pi^*D= \sum\nolimits_{j} a_j \pi^*C_j=a_i E + \sum\nolimits_{j} a_j C'_j.
    $$
    Then the numerical criterion is immediate from the intersection identities, as well as $[D']^2=[D]^2$ and $c_1(\widetilde{X})\cdot D'=c_1(X)\cdot D$.
    If the blow-up is at a node $C_i \cap C_k$, then 
    $$
    D'= \pi^*D= \sum\nolimits_{j} a_j \pi^*C_j=(a_i + a_k) E + \sum\nolimits_{j} a_j C'_j.
    $$
    The rest of the proof is analogous.
\end{proof}

\begin{remark}\label{rmk:intersection_invariant}
    In the following it will also be important for us how the intersection profile of a divisor $D=\sum_j a_j C_j$ with a normal crossing divisor $\cc=\cup_j C_j$ changes under blow-ups. 
    Define the intersection profile to be the vector $\bm{\xi}$ with entries $\xi_j:= D \cdot C_j$, or equivalently $\bm{\xi}=M_\cc \bm{a}$, where $M_\cc :=(C_i\cdot C_j)_{i,j}$ is the intersection matrix of $\cc$.
    If we denote by $\cc'$ the normal crossing divisor obtained by one of the two kinds of blow-ups and insert the new exceptional divisor at the end, as above, the new intersection profile will be $\bm{\xi}'=M_{\cc'} \bm{a}'$, where $\bm{\xi}'=(\xi_1,\ldots,\xi_m,0)$.
\end{remark}

\begin{proof}[Proof of \cref{thm:quadric_complement}.]
    By \cref{thm:structure_regulation} there exists a configuration of symplectic spheres 
    $$
    \ct=(\cf_1,C_{i_\sect},\cf_2) \subseteq \widetilde{X},
    $$
    where $\cc=(C_1,\ldots,C_m) \subseteq \ct$ represents the interface between $U$ and $X\setminus U$.
    Moreover, the broken rulings $\cf_j$ blow down to two smooth fibres and this blow-down procedure can be arranged in such a way that the resulting configuration is $\cd=(F_1,\Sigma_-,F_2) \subseteq \C P^1 \times \C P^1$, where the $F_j$ are smooth fibres and $\Sigma_-$ is a section that squares to $-2$.
    The divisor $D=F_1 + \Sigma_- + F_2$ obviously satisfies the numerical criterion of \cref{thm:neigh_divisor_smoothing}:
    $$
    D \cdot F_j = 1 \quad\text{and}\quad D \cdot \Sigma_-=0.
    $$
    
    As an aside: we therefore obtain a positive section $\Sigma_+$ in a neighbourhood of the chain $\cd$.\footnote{Note that in this case one could have also constructed the $(+2)$-sphere "explicitly" as shown in a similar situation in \cref{fig:quadric_complement}.}
    However, it is unclear if this positive section survives the sequence of symplectic blow-ups.
    
    By \cref{lma:num_criterion_under_blow} and \cref{rmk:intersection_invariant} there exists a divisor, again denoted by $D$, in $\widetilde{X}$
    $$
    D = \sum\nolimits_j a_j C_j + b_1 E_1 + b_2 E_2
    $$
    that satisfies the numerical criterion of \cref{thm:neigh_divisor_smoothing} and satisfies $[D]^2=+2$.
    Therefore, \cref{thm:neigh_divisor_smoothing} constructs a symplectic hypersurface $S \subseteq \nu(\ct)$ that is disjoint from $\cc$ and intersects $E_j$ positively once.
    Moreover, \cref{rmk:neigh_divisor_smoothing_connectedness} implies that $S$ is connected, because $D^2=2$ and $\ct$ is connected, and because $S^2=+2$ and $c_1(\widetilde{X})\cdot S=+4$ the adjunction formula implies that $S$ is an embedded sphere.
    Hence, $S \subseteq \nu(\ct)\setminus \cc \subseteq X \setminus U$ is the embedded symplectic $(+2)$-sphere that we were looking for.
\end{proof}

\begin{remark}
    \cref{thm:quadric_complement} can also be proven with a more "classical" neck stretching approach.
    Note that the analogous theorem to \cref{thm:quadric_complement} in the case of regular ellipsoids is the theorem that shows that embedding ellipsoids into $\C P^2(1)$ is the same as embedding ellipsoids into $B^4(1)$.
\end{remark}

\section{Proving \texorpdfstring{$\ce_{2;1,1}=\ce^\sing_{2;1,1}$}{E211=Esing211}}
\label{sec:E211=E211sing}

We will show that $\ce_{2;1,1}=\ce^\sing_{2;1,1}$ by showing that embeddings of singular ellipsoids and $(2;1,1)$-ellipsoids can be normalised along their respective cores. 
In the case of singular ellipsoids this was done in \cref{lma:normalise_sing_emb} and in the case of $(2;1,1)$-ellipsoids this is a consequence of the nearby Lagrangian conjecture for $S^2$.
This means that we can, for the purpose of embedding problems, pass freely between the $A_1$-surface singularity and its smoothing $T^*S^2$.

\begin{lemma}\label{lma:normalise_cotangent}
    Assume that $f:E_{2;1,1}(\alpha,\beta) \hookrightarrow B_{2;1,1}(1)$ is a symplectic embedding.
    Then, after pre-composing with a symplectomorphism of the domain, there exists a Hamiltonian diffeomorphism $\psi \in \ham(B_{2;1,1}(1),\omega_\can)$ such that $\psi \circ f|_{B_{2;1,1}(\varepsilon)}$ agrees with the visible embedding $\iota_\vis$ for $\varepsilon >0$ sufficiently small.
\end{lemma}

\begin{proof}
    Choose $\delta>0$ small enough that $B_{2;1,1}(\delta) \subseteq E_{2;1,1}(\alpha,\beta)$.
    Then we can view $f$ as a symplectic embedding $f:D^*(S^2,\frac{\tau}{\delta}\lvert \cdot \rvert_{g_0^*}) \hookrightarrow D^*(S^2,\tau\lvert \cdot \rvert_{g^*_0})$.
    Lagrangian spheres in $D^*(S^2,\tau \lvert \cdot \rvert_{g^*_0})$ are unique up to Hamiltonian isotopy by \cite{Hi12,LiWu12} and therefore there exists a Hamiltonian diffeomorphism $\phi$ such that $\phi \circ f$ is the identity on $0_{S^2}$, after possibly pre-composing with a symplectomorphism of the domain. 
    Now, by the relative Weinstein--Moser trick we can find a Hamiltonian diffeomorphism $\varphi$, supported in a small neighbourhood of $0_{S^2}$, such that $g:=\varphi \circ \phi \circ f$ is the identity on $D^*(S^2,\frac{\tau}{\varepsilon}\lvert \cdot \rvert_{g^*_0})$ for $\varepsilon >0$ small enough and hence $\psi:=\varphi \circ \phi$ is the Hamiltonian diffeomorphism we were looking for.
    In particular, this means that $g$ coincides with the visible embedding $\iota_\vis$ on $B_{2;1,1}(\varepsilon)$.
\end{proof}

\begin{corollary}
    We have that $\ce_{2;1,1}=\ce^\sing_{2;1,1}$.
\end{corollary}

\begin{proof}
    This is a direct consequence of \cref{lma:normalise_sing_emb} and \cref{lma:normalise_cotangent}.
    Assume that
    $$
    f:E^\sing_{2;1,1}(\alpha,\beta) \xhookrightarrow{s} B^\sing_{2;1,1}(1)
    $$
    is an embedding.
    By \cref{lma:normalise_sing_emb} we can assume that it coincides with the inclusion on $B^\sing_{2;1,1}(\varepsilon)$ for $\varepsilon >0$ sufficiently small.
    This means we obtain an embedding 
    $$
    \widehat{f}:E_{2;1,1}(\alpha,\beta) \xhookrightarrow{s} B_{2;1,1}(1)
    $$
    by gluing $\iota_\vis$ on $B_{2;1,1}(\varepsilon)$ with $f$.
    This is done by using the symplectomorphisms:
    $$
    E^\sing_{2;1,1}(\alpha,\beta)\setminus B^\sing_{2;1,1}(\varepsilon) \to E_{2;1,1}(\alpha,\beta)\setminus B_{2;1,1}(\varepsilon)
    $$
    and the analogous one on the target.
    These symplectomorphisms exist because the moment maps on these regions coincide. See for example \cite[Chapter 2]{Ev23:book}.
    The other direction follows analogously.
\end{proof}

\begin{remark}
    It is clear from the method of proof that the results of this section can be extended to the general $(p,q)$-case by using the results in \cite{AdaBaHau26}.
\end{remark}

\section{Some comments on the \texorpdfstring{$\R P^2$}{RP2} case}
\label{sec:RP2_case}

As mentioned in the introduction, the staircase for the $(2,1)$-embedding problem, i.e.\ the computation of the set $\ca_{2,1}$, introduced in \eqref{eq:def_A21}, was proven in \cite[Theorem 1.5.2]{ABEHS25}, which corresponds to \cref{sec:A211_staircase} here.
The interpretation of the $E_{2,1}(\alpha,\beta)$ in terms of codisc bundles with respect to Randers metrics was discussed in \cref{rmk:Randers_interpretation_RP2}.
Proving that $\ca_{2,1}=\ce_{2,1}$ follows by arguments analogous to the ones given in \cref{sec:A211=E211}.
The difference is that the open unit codisc bundle of $\R P^2$ with respect to the round metric compactifies to $\C P^2$, see for example \cite{Ada25}, and so the analogue of \cref{thm:quadric_complement} is that in the complement of an embedded $E_{2,1}(\alpha,\beta)$ in $\C P^2(2)$ there is always an embedded symplectic $(+4)$-sphere.
Uniqueness of such spheres in $\C P^2$ then proves $\ca_{2,1}=\ce_{2,1}$.
Finally, the arguments in \cref{sec:E211=E211sing} again follow by the same techniques but one has to appeal to the uniqueness of Lagrangian $\R P^2$s in $T^*\R P^2$, which is proven in \cite{HiPiWu16,AdaBaHau26,Ada25}.
This then shows $\ce_{2,1}=\ce^\sing_{2,1}=\ce^\eq_{2,1}$, where the last equality follows from \cref{lma:eq=sing}.

\emergencystretch=2.5em
\printbibliography

@article{AdHa25,
author = {Adaloglou, N. and Hauber, J.},
title = {Pinwheels in symplectic rational and ruled surfaces and non-squeezing of rational homology balls},
journal = {Journal of Topology},
volume = {19},
number = {3},
pages = {e70094},
doi = {https://doi.org/10.1112/topo.70094},
url = {https://londmathsoc.onlinelibrary.wiley.com/doi/abs/10.1112/topo.70094},
year = {2026}
}

@misc{ABEHS25,
    title={Markov staircases}, 
    author={N. Adaloglou and J. Brendel and J. Evans and J. Hauber and F. Schlenk},
    year={2025},
    eprint={2509.03224},
    archivePrefix={arXiv},
    primaryClass={math.SG},
    url={https://arxiv.org/abs/2509.03224}, 
}

@article{Ra04,
    author = {Rademacher, H.-B.},
    title = {A sphere theorem for non-reversible {Finsler} metrics},
    fjournal = {Mathematische Annalen},
    journal = {Math. Ann.},
    issn = {0025-5831},
    volume = {328},
    number = {3},
    pages = {373--387},
    year = {2004},
    language = {English},
    doi = {10.1007/s00208-003-0485-y},
    zbMATH = {2078020},
    Zbl = {1050.53063}
}

@article{Ka74,
    author = {Katok, A. B.},
    title = {Ergodic perturbations of degenerate integrable {Hamiltonian} systems},
    fjournal = {Mathematics of the USSR. Izvestiya},
    journal = {Math. USSR, Izv.},
    issn = {0025-5726},
    volume = {7},
    pages = {535--571},
    year = {1974},
    language = {English},
    doi = {10.1070/IM1973v007n03ABEH001958},
    zbMATH = {3495405},
    Zbl = {0316.58010}
}

@article{Zi83,
    author = {Ziller, W.},
    title = {Geometry of the {Katok} examples},
    fjournal = {Ergodic Theory and Dynamical Systems},
    journal = {Ergodic Theory Dyn. Syst.},
    issn = {0143-3857},
    volume = {3},
    pages = {135--157},
    year = {1983},
    language = {English},
    doi = {10.1017/S0143385700001851},
    zbMATH = {3892261},
    Zbl = {0559.58027}
}

@article{Sh02,
    author = {Shen, Z.},
    title = {Two-dimensional {Finsler} metrics with constant flag curvature},
    fjournal = {Manuscripta Mathematica},
    journal = {Manuscr. Math.},
    issn = {0025-2611},
    volume = {109},
    number = {3},
    pages = {349--366},
    year = {2002},
    language = {English},
    doi = {10.1007/s00229-002-0311-y},
    zbMATH = {1889165},
    Zbl = {1027.53093}
}

@article{HaPa08,
    author = {Harris, A. and Paternain, G.},
    title = {Dynamically convex {Finsler} metrics and {{\(J\)}}-holomorphic embedding of asymptotic cylinders},
    fjournal = {Annals of Global Analysis and Geometry},
    journal = {Ann. Global Anal. Geom.},
    issn = {0232-704X},
    volume = {34},
    number = {2},
    pages = {115--134},
    year = {2008},
    language = {English},
    doi = {10.1007/s10455-008-9111-2},
    zbMATH = {5324305},
    Zbl = {1149.53045}
}

@article{BaRoSh04,
    author = {Bao, D. and Robles, C. and Shen, Z.},
    title = {Zermelo navigation on {Riemannian} manifolds},
    fjournal = {Journal of Differential Geometry},
    journal = {J. Differ. Geom.},
    issn = {0022-040X},
    volume = {66},
    number = {3},
    pages = {377--435},
    year = {2004},
    language = {English},
    doi = {10.4310/jdg/1098137838},
    zbMATH = {2143427},
    Zbl = {1078.53073}
}

@book{Ev23:book,
    author = {Evans, J.},
    title = {Lectures on {Lagrangian} torus fibrations},
    fseries = {London Mathematical Society Student Texts},
    series = {Lond. Math. Soc. Stud. Texts},
    issn = {0963-1631},
    volume = {105},
    isbn = {978-1-00-937263-3},
    year = {2023},
    publisher = {Cambridge: Cambridge University Press},
    language = {English},
    doi = {10.1017/9781009372671},
    zbMATH = {7689805},
    Zbl = {1528.53001}
}

@incollection{Sym03:four_two,
    author = {Symington, M.},
    title = {Four dimensions from two in symplectic topology},
    booktitle = {Topology and geometry of manifolds. Proceedings of the 2001 Georgia topology conference, University of Georgia, Athens, GA, USA, May 21--June 2, 2001},
    isbn = {0-8218-3507-6},
    pages = {153--208},
    year = {2003},
    publisher = {Providence, RI: American Mathematical Society (AMS)},
    language = {English},
    zbMATH = {2065294},
    Zbl = {1049.57016}
}

@article {Gro85,
    AUTHOR = {Gromov, M.},
    TITLE = {Pseudo holomorphic curves in symplectic manifolds},
    JOURNAL = {Invent. Math.},
    FJOURNAL = {Inventiones Mathematicae},
    VOLUME = {82},
    YEAR = {1985},
    NUMBER = {2},
    PAGES = {307--347},
    ISSN = {0020-9910},
    MRCLASS = {53C15 (32F25 53C57 57R15)},
    MRNUMBER = {809718},
    MRREVIEWER = {Yakov Eliashberg},
    DOI = {10.1007/BF01388806},
    URL = {https://doi.org/10.1007/BF01388806},
}

@article{Aud07,
     author = {Audin, M.},
     title = {Lagrangian skeletons, periodic geodesic flows and symplectic cuttings},
     fjournal = {Manuscripta Mathematica},
     journal = {Manuscr. Math.},
     issn = {0025-2611},
     volume = {124},
     number = {4},
     pages = {533--550},
     year = {2007},
     language = {English},
     doi = {10.1007/s00229-007-0134-y},
     zbMATH = {5248331},
     Zbl = {1132.53042}
}

@article{OaUs16,
    author = {Oakley, J. and Usher, M.},
    title = {On certain {Lagrangian} submanifolds of {{\(S^2\times S^2\)}} and {{\(\mathbb{C}\operatorname{P}^n\)}}},
    fjournal = {Algebraic \& Geometric Topology},
    journal = {Algebr. Geom. Topol.},
    issn = {1472-2747},
    volume = {16},
    number = {1},
    pages = {149--209},
    year = {2016},
    language = {English},
    doi = {10.2140/agt.2016.16.149},
    zbMATH = {6553357},
    Zbl = {1335.53105}
}

@misc{MuRo25,
    author = {Mu{\~n}oz, V. and Rojo, J.},
    title = {Constructions of symplectic surfaces in symplectic 4-manifolds with transversal intersections},
    year = {2025},
    howpublished = {Preprint, {arXiv}:2503.12428 [math.{SG}] (2025)},
    url = {https://arxiv.org/abs/2503.12428},
    arXiv = {arXiv:2503.12428}
}

@misc{Ev24:KIAS,
    author = {Evans, J.},
    title = {{KIAS} {Lectures} on {Symplectic} {Aspects} of {Degenerations}},
    year = {2024},
    howpublished = {Preprint, {arXiv}:2403.03519 [math.{SG}]},
    url = {https://arxiv.org/abs/2403.03519},
    arXiv = {arXiv:2403.03519}
}

@article{HiPiWu16,
    author = {R. Hind and M. Pinsonnault and W. Wu},
    title = {Symplectormophism groups of non-compact manifolds, orbifold balls, and a space of {Lagrangians}},
    fjournal = {The Journal of Symplectic Geometry},
    journal = {J. Symplectic Geom.},
    issn = {1527-5256},
    volume = {14},
    number = {1},
    pages = {203--226},
    year = {2016},
    language = {English},
    doi = {10.4310/JSG.2016.v14.n1.a8},
    zbMATH = {6623433},
    Zbl = {1355.57026}
}

@book{BaChSh00,
    author = {Bao, D. and Chern, S.-S. and Shen, Z.},
    title = {An introduction to {Riemann}-{Finsler} geometry},
    fseries = {Graduate Texts in Mathematics},
    series = {Grad. Texts Math.},
    issn = {0072-5285},
    volume = {200},
    isbn = {0-387-98948-X},
    year = {2000},
    publisher = {New York, NY: Springer},
    language = {English},
    zbMATH = {1001612},
    Zbl = {0954.53001}
}

@article{AbSaSch23,
    author = {Abbondandolo, A. and Alves, Marcelo R. R. and Sa{\u{g}}lam, M. and Schlenk, F.},
    title = {Entropy collapse versus entropy rigidity for {Reeb} and {Finsler} flows},
    fjournal = {Selecta Mathematica. New Series},
    journal = {Sel. Math., New Ser.},
    issn = {1022-1824},
    volume = {29},
    number = {5},
    pages = {99},
    note = {Id/No 67},
    year = {2023},
    language = {English},
    doi = {10.1007/s00029-023-00865-8},
    zbMATH = {7733768},
    Zbl = {1527.37037}
}

@article{Ch96,
    author = {Chern, S.-S.},
    title = {Finsler geometry is just {Riemannian} geometry without the quadratic restriction.},
    fjournal = {Notices of the American Mathematical Society},
    journal = {Notices Am. Math. Soc.},
    issn = {0002-9920},
    volume = {43},
    number = {9},
    pages = {959--963},
    year = {1996},
    language = {English},
    zbMATH = {1700702},
    Zbl = {1044.53512}
}

@book{Ber03,
    author = {Berger, M.},
    title = {A panoramic view of {Riemannian} geometry},
    isbn = {3-540-65317-1},
    year = {2003},
    publisher = {Berlin: Springer},
    language = {English},
    doi = {10.1007/978-3-642-18245-7},
    zbMATH = {1973376},
    Zbl = {1038.53002}
}

@article{McDSie25,
    author = {McDuff, D. and Siegel, K.},
    title = {Singular algebraic curves and infinite symplectic staircases},
    fjournal = {Inventiones Mathematicae},
    journal = {Invent. Math.},
    issn = {0020-9910},
    volume = {242},
    number = {2},
    pages = {387--459},
    year = {2025},
    language = {English},
    doi = {10.1007/s00222-025-01359-4},
    zbMATH = {8103859}
}

@article{Hi12,
    author = {Hind, R.},
    title = {Lagrangian unknottedness in {Stein} surfaces},
    fjournal = {The Asian Journal of Mathematics},
    journal = {Asian J. Math.},
    issn = {1093-6106},
    volume = {16},
    number = {1},
    pages = {1--36},
    year = {2012},
    language = {English},
    doi = {10.4310/AJM.2012.v16.n1.a1},
    zbMATH = {6050962},
    Zbl = {1262.53073},
}

@article{LiWu12,
    author = {Li, T.-J. and Wu, W.},
    title = {Lagrangian spheres, symplectic surfaces and the symplectic mapping class group},
    fjournal = {Geometry \& Topology},
    journal = {Geom. Topol.},
    issn = {1465-3060},
    volume = {16},
    number = {2},
    pages = {1121--1169},
    year = {2012},
    language = {English},
    doi = {10.2140/gt.2012.16.1121},
    zbMATH = {6068622},
    Zbl = {1253.53073},
}

@misc{AdaBaHau26,
    author = {Adaloglou, N. and Bargall{\'o} i G{\'o}mez, G. and Hauber, J.},
    title = {The nearby {Lagrangian} conjecture for pinwheels},
    year = {2026},
    howpublished = {Preprint, {arXiv}:2605.22473 [math.{SG}] (2026)},
    url = {https://arxiv.org/abs/2605.22473},
    arXiv = {arXiv:2605.22473}
}

@article{McDSch12,
    author = {McDuff, D. and Schlenk, F.},
    title = {The embedding capacity of 4-dimensional symplectic ellipsoids},
    fjournal = {Annals of Mathematics. Second Series},
    journal = {Ann. Math. (2)},
    issn = {0003-486X},
    volume = {175},
    number = {3},
    pages = {1191--1282},
    year = {2012},
    language = {English},
    doi = {10.4007/annals.2012.175.3.5},
    zbMATH = {6051270},
    Zbl = {1254.53111}
}

@article{Sch18,
    author = {Schlenk, F.},
    title = {Symplectic embedding problems, old and new},
    fjournal = {Bulletin of the American Mathematical Society. New Series},
    journal = {Bull. Am. Math. Soc., New Ser.},
    issn = {0273-0979},
    volume = {55},
    number = {2},
    pages = {139--182},
    year = {2018},
    language = {English},
    doi = {10.1090/bull/1587},
    zbMATH = {7175361},
    Zbl = {1464.53106}
}

@article{Ada25,
    author = {Adaloglou, N.},
    title = {Uniqueness of {Lagrangians} in {{\(T^*{\mathbb{R}}P^2\)}}},
    fjournal = {Annales Math{\'e}matiques du Qu{\'e}bec},
    journal = {Ann. Math. Qu{\'e}.},
    issn = {2195-4755},
    volume = {49},
    number = {1},
    pages = {215--222},
    year = {2025},
    language = {English},
    doi = {10.1007/s40316-024-00238-3},
    zbMATH = {8030618},
    Zbl = {1564.53074}
}

@misc{Car19,
    author = {F. C. Caramello},
    title = {Introduction to orbifolds},
    year = {2019},
    howpublished = {Preprint, {arXiv}:1909.08699 [math.{DG}] (2019)},
    url = {https://arxiv.org/abs/1909.08699},
    arXiv = {arXiv:1909.08699}
}

@book{AdLeRu07,
    author = {Adem, A. and Leida, J. and Ruan, Y.},
    title = {Orbifolds and stringy topology},
    fseries = {Cambridge Tracts in Mathematics},
    series = {Camb. Tracts Math.},
    issn = {0950-6284},
    volume = {171},
    isbn = {978-0-521-87004-7},
    year = {2007},
    publisher = {Cambridge: Cambridge University Press},
    language = {English},
    zbMATH = {5234797},
    Zbl = {1157.57001}
}

@article{Fer24,
    author = {Ferreira, B.},
    title = {Elliptic {Reeb} orbit on some real projective three-spaces via {ECH}},
    fjournal = {Revista Matem{\'a}tica Iberoamericana},
    journal = {Rev. Mat. Iberoam.},
    issn = {0213-2230},
    volume = {40},
    number = {5},
    pages = {1833--1862},
    year = {2024},
    language = {English},
    doi = {10.4171/RMI/1480},
    zbMATH = {7927755},
    Zbl = {1555.53131}
}

@article{FerRam22,
    author = {Ferreira, B. and Ramos, V. G. B.},
    title = {Symplectic embeddings into disk cotangent bundles},
    fjournal = {Journal of Fixed Point Theory and Applications},
    journal = {J. Fixed Point Theory Appl.},
    issn = {1661-7738},
    volume = {24},
    number = {3},
    pages = {31},
    note = {Id/No 62},
    year = {2022},
    language = {English},
    doi = {10.1007/s11784-022-00979-0},
    zbMATH = {7591740},
    Zbl = {1529.53081}
}

@unpublished{HaSch26,
    author = {Hauber, J. and Schlenk, F. and Ramos, V. G. B.},
    title  = {The Gromov width of Katok ellipsoids},
    note   = {Forthcoming}
}

@book{McDSal16,
    author = {McDuff, D. and Salamon, D.},
    title = {Introduction to symplectic topology},
    edition = {3rd edition},
    fseries = {Oxford Graduate Texts in Mathematics},
    series = {Oxf. Grad. Texts Math.},
    volume = {27},
    isbn = {978-0-19-879489-9},
    year = {2016},
    publisher = {Oxford: Oxford University Press},
    language = {English},
    zbMATH = {6638013},
    Zbl = {1380.53003}
}

\end{document}